\documentclass[12pt]{amsart}
\usepackage{placeins}
\usepackage{mathtools}
\usepackage{hyperref}
\usepackage{cleveref}
\usepackage{subcaption}
\usepackage{tikz}
\usepackage{float}
\usetikzlibrary{calc} 
\usepackage{amsthm, amsmath}

    \newtheorem{theorem}{Theorem}[section]
    \newtheorem{lemma}[theorem]{Lemma}
    \newtheorem{corollary}[theorem]{Corollary}

    \newtheorem{example}[theorem]{Example}
    
    \newtheorem{remark}{Remark}

    \newtheorem{comment}{Comment}
    \newtheorem{definition}[theorem]{Definition}

\usepackage{amsmath}
\usepackage{amssymb}
\usepackage[top=1in, left=1in, right=1in, bottom=1in]{geometry}
\hypersetup{colorlinks,allcolors=blue}
\usepackage{etoolbox}
\usepackage{tikz}
\usepackage{float}

\usepackage[utf8]{inputenc}
\usepackage{graphicx}
\usepackage{gensymb}
\usepackage{xcolor}
\usepackage[T1]{fontenc}
\usepackage{array}
\usepackage[export]{adjustbox}
\usepackage{comment}
\usepackage[dvipsnames]{xcolor}

\newcommand{\ar}[1]{\overset{\rightarrow}{#1}}
\newcommand{\AR}[1]{\overrightarrow{#1}}

\usepackage{tikz}
\usepackage{subcaption}
\usetikzlibrary{external,quotes,chains}
\tikzset{vertex/.style={black,fill,draw,minimum size=6pt,inner sep=0pt,circle,thin},bold/.style={black,line width=0.6mm},plain/.style={black,thin},bold edges/.style=bold,plain edges/.style=plain,label distance=1mm,text node/.style={rectangle,fill=none,draw=none},every label/.style=text node,caption node/.style={text node,font=\Large}}

\title{Ornaments and Difference Distance Magic Oriented Graphs}
\author[R.~Aceska]{Roza Aceska*} 
\address{Department of Mathematical Sciences, Ball State University, Muncie, IN, USA}
\email{raceska@bsu.edu }
\author[M.~Lort]{McKailyn Lort}
\address{Department of Mathematical Sciences, Ball State University, Muncie, IN, USA}
\email{mckailyn.lort@bsu.edu}

\author[A.~Ripperger]{Allison Ripperger}
\address{Department of Mathematical Sciences, Ball State University, Muncie, IN, USA}
\email{allison.ripperger@bsu.edu}
\date{}

\begin{document}

\maketitle
\begin{abstract}
One way to construct difference distance magic oriented graphs  is via weighted sums, a technique initially introduced in \cite{AAetc}. We explore the quality of said construction further by introducing the notion of an  ornament. An ornament is an oriented graph that, when  used in a weighted sum with  an existing difference distance magic oriented graph (DDMOG),  creates a new DDMOG. We provide  results on the construction of a specific type of ornaments, called $s-$nodes ornaments.  
        We conclude the paper with a list of open questions related to DDMOGs and ornaments.
\end{abstract}

\section{Introduction}
Graph labeling problems, which involve assigning integers to vertices or edges to satisfy specific sum-based properties, have been studied extensively since Sedláček posed a question on magic labeling in  \cite{Sed}; notable variations of labelings include super edge-magic labelings \cite{Enomoto} and harmonious labelings \cite{Graham}.  
For a comprehensive survey of results on various graph labelings we refer  to the dynamic survey by Gallian \cite{Gal}.  
In \cite{FM}, difference distance magic (DDM) labeling was introduced on oriented graphs, with initial examples of graphs with DDM labeling, called difference distance magic oriented graphs (DDMOGs). At every vertex of a DDMOG, the sum of the labels of its in-neighbors (the incoming flow) exactly matches the sum of the labels of its out-neighbors (the outgoing flow). This property can be seen as a direct discrete analog to Kirchhoff’s Current Law, which states that the total current entering a junction (node) in an electrical circuit exactly equals the total current leaving that junction. Thus, the DDM labeling ensures that the ``potential'' provided by the vertex labels is perfectly distributed across the oriented edges.

In this paper we explore  the construction of new oriented graphs with DDM labeling from smaller, well-selected oriented graphs. We offer a systemized method to {\it{attach}} oriented cycles to an existing DDMOG, thereby creating a new DDMOG with a larger number of vertices. 

 \subsection{Definitions and known results}
In this paper, all graphs $G=(V(G), E(G))$ are finite, undirected, and simple, where $n=|V(G)|$ and $m=|E(G)|$.  
 The \emph{neighborhood} of a vertex is the set of its neighbors, and given a vertex $v \in V(G)$ is denoted $N(v)$.  The \emph{degree} of a vertex $v$ is $|N(v)|$. A \emph{digraph} $D$ is a graph with directed edges: it is defined to be a pair $(V,E)$, consisting of a set of \emph{vertices} $V(D)$ and a collection of ordered pairs of vertices $E(D)$, called \emph{edges}. If $u, v \in V(D)$ and $(u,v) \in E(D)$, then we say $u$ and $v$ are \emph{adjacent} (or are \emph{neighbors}) and that the edge $(u,v)$ is \emph{incident} to $u$ and $v$, and is illustrated in diagrams by $u\bullet$ \hspace{-0.1in} $\longrightarrow$ \hspace{-0.1in} $\bullet v$.  The \emph{in-neighborhood} of a vertex $v$ is the set of vertices $u \in V(D)$ such that $(u,v) \in E(D)$ and is denoted $N^+(v)$.  The \emph{out-neighborhood} of a vertex $v$ is the set of vertices $u \in V(D)$ such that $(v,u) \in E(D)$ and is denoted $N^-(v)$. When necessary to avoid confusion, we use $N_G^+(v)$ and $N_G^-(v)$ to denote the in- and out-neighborhoods of $v$, respectively, in a specified graph $G$. 
An \emph{oriented graph} $\AR{G}$ is a digraph such that $(u,v) \in E(\ar{G})$ implies $(v,u) \not \in E(\ar{G})$. When working with oriented graphs, we use both $(u,v)$ and $uv$ for the oriented edge $u\rightarrow v$. The \emph{underlying graph} $G$ of a digraph $D$ is the graph such that  
        1. $V(G)=V(D)$,  and  2. if $(u,v) \in E(D)$ or $(v,u) \in E(D)$, then $uv \in E(G)$.
               If the underlying graph of a digraph $D$ is $G$, then we say that $D$ is a \emph{direction} of $G$.  Likewise, if the underlying graph of an oriented graph $\AR{G}$ is $G$, then we say that $\AR{G}$ is an \emph{orientation} of $G$.  
  A {\it cycle} $C_L$ (or $L-$cycle) is a graph whose $L$ vertices are $\it v_1, v_2, \ldots, v_L$, and its edges are $\it v_1v_2$, $v_2v_3$, $\ldots$, $v_{L-1}v_L$ and $\it v_Lv_1$. 
    In this paper we only work with oriented cycles, such that at every vertex there is one incoming and one outgoing edge; in Subfigure~\ref{fig:b}  the two copies of $\ar{C_3}$ with labels $6, 7, 8$ and $15, 16, 17$ have the {\it same orientation}, while the two copies of  two copies of $\ar{C_3}$ with labels $6, 7, 8$ and $9, 10, 11$  have {\it opposite orientations}. 

\begin{figure}
\centering

\begin{minipage}[b]{0.45\textwidth}  \centering
\vspace{-5mm}

  \begin{subfigure}[t]{\textwidth}
  \centering
  \begin{tikzpicture}[scale=0.76,node distance=1.4cm, base node/.style={circle,draw,minimum size=21pt}]

\node[base node, very thick] (5) at (1, 6.5) {5}; 
\node[base node, very thick] (1) at (-1, 5) {1};
\node[base node, very thick] (2) at (3, 5) {2};
\node[base node, very thick] (3) at (-1,8) {3};
\node[base node, very thick] (4) at (3,8) {4};

\path[->, very thick] (5) edge (1);
\path[->, very thick] (2) edge (5);
\path[->, very thick] (3) edge  (5); %
\path[->, very thick] (5) edge  (4); %

\path[->, very thick] (1) edge[bend right=27]   (2);
\path[->, very thick] (1) edge [bend left=39]  (3);
\path[->, very thick] (4) edge[bend left=27]  (2);
\path[->, very thick] (4) edge[bend right=27]   (3);

\end{tikzpicture}
  \caption{ DDM Labeling of $\overrightarrow{W_4}$}
  \label{fig:a}
  \end{subfigure}
	
  \vspace{0.5cm} 
	
  \begin{subfigure}[b][.21\textheight][t]{\textwidth}  \centering
  \begin{tikzpicture}[scale=0.76,node distance=1.4cm, base node/.style={circle,draw,minimum size=21pt}]

 \node[base node, very thick,  label={[text=teal] left:+1}] (6) at (-2, 2.5) {6};
\node[base node, very thick, label={[text=orange] below:-2}] (7) at (0, 2.5) {7};
\node[base node, very thick, label={[text=teal] left:+1}] (8) at (-2, .5) {8};
\node[base node, very thick, label={[text=teal] left:-1}] (9) at (3, 2.5) {9};
 \node[base node, very thick, label={[text=orange] below:+2}] (10) at (5, 2.5) {10};
\node[base node, very thick, label={[text=teal] left:-1}] (11) at (3, .5) {11};
 \path[->, very thick] (6) edge (7);
\path[->, very thick] (7) edge(8); 
  \path[->, very thick] (8) edge (6);
\path[->, very thick] (9) edge (11);
\path[->, very thick] (11) edge  (10); %
\path[->, very thick] (10) edge  (9); %

 \node[base node, very thick, label={[text=teal] left:-1}] (12) at (-2, -1) {12};
\node[base node, very thick, label={[text=orange] below:+2}] (13) at (0, -1) {13};
\node[base node, very thick, label={[text=teal] left:-1}] (14) at (-2, -3) {14};
\node[base node, very thick, label={[text=teal] left:+1}] (15) at (3, -1) {15};
 \node[base node, very thick, label={[text=orange] below:-2}] (16) at (5, -1) {16};
\node[base node, very thick, label={[text=teal] left:+1}] (17) at (3, -3) {17};
 \path[->, very thick] (13) edge (12);
\path[->, very thick] (14) edge(13); 
  \path[->, very thick] (12) edge (14);
\path[->, very thick] (15) edge (16);
\path[->, very thick] (16) edge  (17); %
\path[->, very thick] (17) edge  (15); %

\end{tikzpicture}
  \caption{Labeling of  $4\overrightarrow{C_3}$ (weights \textcolor{teal}{$\pm1$}, \textcolor{orange}{$\pm 2$})}
  \label{fig:b}
  \end{subfigure}
  \vspace{13mm}
\end{minipage}
\hfill

\begin{minipage}[b]{0.45\textwidth}
\begin{center}
  \begin{subfigure}[t]{\textwidth}
  \centering
  \begin{tikzpicture}[scale=0.76,node distance=1.4cm, base node/.style={circle,draw,minimum size=21pt}]
 \node[base node, very thick] (6) at (-2, .5) {6};
\node[base node, very thick] (7) at (0, .5) {7};
\node[base node, very thick] (8) at (-2, -1.5) {8};
\node[base node, very thick] (9) at (3, .5) {9};
 \node[base node, very thick] (10) at (5, .5) {10};
\node[base node, very thick] (11) at (3, -1.5) {11};
 \path[->, very thick] (6) edge (7);
\path[->, very thick] (7) edge(8); 
  \path[->, very thick] (8) edge (6);
\path[->, very thick] (9) edge (11);
\path[->, very thick] (11) edge  (10); %
\path[->, very thick] (10) edge  (9); %

 \node[base node, very thick] (12) at (-2, -3) {12};
\node[base node, very thick] (13) at (0, -3) {13};
\node[base node, very thick] (14) at (-2, -5) {14};
\node[base node, very thick] (15) at (3, -3) {15};
 \node[base node, very thick] (16) at (5, -3) {16};
\node[base node, very thick] (17) at (3, -5) {17};
 \path[->, very thick] (13) edge (12);
\path[->, very thick] (14) edge(13); 
  \path[->, very thick] (12) edge (14);
\path[->, very thick] (15) edge (16);
\path[->, very thick] (16) edge  (17); %
\path[->, very thick] (17) edge  (15); %

\node[base node, very thick] (5) at (1, 4.5) {5}; 
\node[base node, very thick] (1) at (-1, 3) {1};
\node[base node, very thick] (2) at (3, 3) {2};
\node[base node, very thick] (3) at (-1,6) {3};
\node[base node, very thick] (4) at (3,6) {4};

\path[->, very thick] (5) edge (1);
\path[->, very thick] (2) edge (5);
\path[->, very thick] (3) edge  (5); %
\path[->, very thick] (5) edge  (4); %

\path[->, very thick] (1) edge[bend right=27]   (2);
\path[->, very thick] (1) edge [bend left=39]  (3);
\path[->, very thick] (4) edge[bend left=27]  (2);
\path[->, very thick] (4) edge[bend right=27]   (3);
\path[->, very thick, orange] (10) edge (2);
\path[->, very thick, orange] (13) edge (2);
\path[->, very thick, orange] (2) edge (7);
\path[->, very thick, orange] (2) edge[bend left=52]  (16);
\path[->, very thick, teal] (6) edge (1);
\path[->, very thick,  teal] (1) edge (9);
\path[->, very thick,  teal] (1) edge[bend right=37]  (12); %
\path[->, very thick,  teal] (15) edge[bend right=3]  
(1); %

\path[->, very thick,  teal] (8) edge[bend left=58]  (1);
\path[->, very thick,  teal] (1) edge (11);
\path[->, very thick,  teal] (1) edge[bend right=39]  (14); %
\path[->, very thick,  teal] (17) edge[bend right=10]  (1); %

\end{tikzpicture}
  \caption{	DDM Labeling of $\overrightarrow{W_4} \oplus 4   \overrightarrow{C_3}$.}
  \label{fig:c}
  \end{subfigure}\end{center}
		\vspace{-1mm}
\end{minipage}

\caption{The construction of  DDMOG $\overrightarrow{W_4} \oplus 4   \overrightarrow{C_3}$ and its DDM labeling, following the proof of Theorem~\ref{thms} (when $s=2$).} \label{fig:exampleTheorem}
\end{figure}
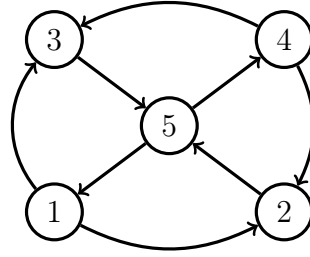
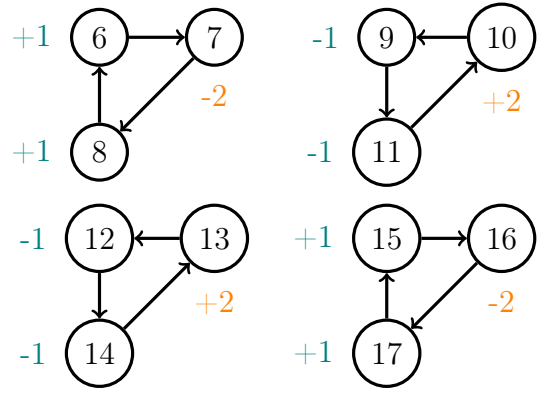
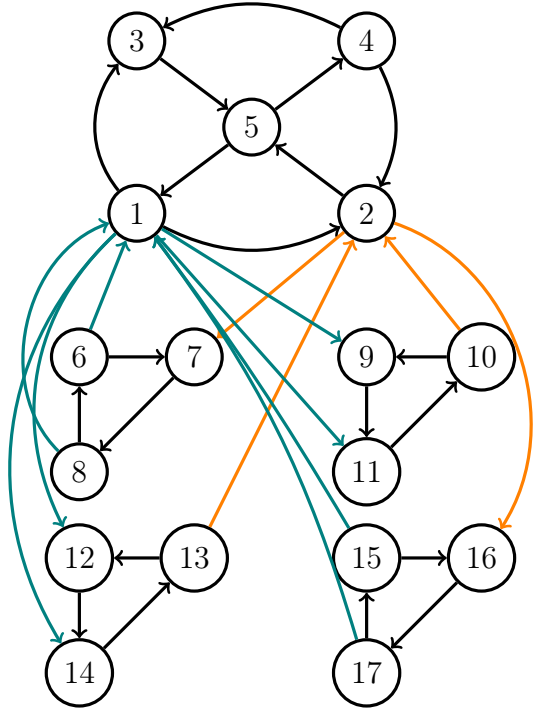

For a graph $G$, a {\it labeling function} is a bijective function $f$ that maps $V(G)$ to some set of integers, called {\it a set of labels}. Let $x\in V(\ar{G})$. The {\it weight} of vertex $x$ is defined \cite{FM} as the sum of the vertex labels directed into $x$ minus the sum of vertex labels directed out of $x$, 
\begin{equation*}
    \text{wt}(x) = \sum_{y\in N^+(x)}f(y)-\sum_{y\in N^-(x)}f(y).
\end{equation*}
 If the labeling function we are working with is not clear, we use $\text{wt}_f(v)$ to denote the weight of the vertex $v$ with labeling $f$.  

A \emph{difference distance magic} (DDM) labeling \cite{FM} of an oriented graph $\ar{G}$ with $n$ vertices is a bijection $f: V(\ar{G})\rightarrow \{1,2,...,n\}$ where there exists some $k\in \mathbb{N}$ such that for all $x\in V(\ar{G})$ $\text{wt}_f(x)=k$; it shows   that $k=0$ for all $x\in V(\ar{G})$.  
     %
     If an oriented graph $\AR{G}$ has a DDM labeling then we say that $\AR{G}$ is a \textit{difference distance magic oriented graph (DDMOG)}. If there exists  an orientation  on a graph $G$ that  $\ar{G}$ is a DDMOG, then we refer to  $G$ as  \textit{difference distance magic orientable (DDMO)}.   It is trivial to see that   any digraph where every edge is bidirectional is trivially a DDMDG. Another trivial example of a DDMOG is the zero edges graph $n\ar{K_1}$ for any $n\geq 1$. 
   It shows \cite{FM} that if an oriented graph is a path or a cycle, then it is not a DDMOG. 
   The {\it smallest} nontrivial example of a graph with DDM labeling is the wheel $\vec{W_4}$ (Figure~\ref{fig:a}); here we refer to the smallest number of vertices. 
   It is known \cite{AAetc} that there exists a $n-$vertices oriented graph with DDM labeling for any $n\geq 5$. 
A new type of a graph sum was introduced in \cite{AAetc}, formed by taking two oriented graphs   and adding additional edges based on the labeling-induced weights at each vertex:
\begin{definition}\label{def:whtsumN}
       Let $k\in \mathbb Z^+$ and let $\ar{G}$ be an oriented graph with an associated labeling function
      $g: V(\ar{G}) \rightarrow \{1,2,\dots, n\}$,
      where $g(v_i)=i$ for any  $v_i \in V(\ar{G})$. If $\ar{H}$ is an oriented graph with labeling $h: V(H) \rightarrow \mathbb{Z}^+$, where  
       $\{ |\text{wt}_h(u)| : u \in V(\ar{H}) \} \subseteq \{0\} \cup \{ g(v)+ks : v \in V(\ar{G})\}$ for some $k \in \mathbb Z$,  then the weighted sum of $\ar{G}$ and $\ar{H}$ shifted by $k$ denoted $\ar{G} \oplus_{\text{wt}_h}^k \ar{H}$
      is the oriented graph with vertex set $V(\AR{G} \oplus_{\text{wt}_h}^k \AR{H})=V(\AR{G}) \cup V(\AR{H})$ and $E(\AR{G} \oplus_{\text{wt}_h}^k \AR{H})=E(\AR{G}) \cup E(\AR{H}) \cup (\bigcup\limits_{i=1}^n (E^i \cup E^{-i}))$, where $E^i=\{v_i u:u \in V_h^{-i-k}(\AR{H})\}$ and $E^{-i}=\{uv_i:u \in V_h^{i+k}(\AR{H})\}$. 
 When $k=0$, we  denote the weighted sum by $\AR{G} \oplus_{\text{wt}_h} \AR{H}$ instead of $\AR{G} \oplus_{\text{wt}_h}^0 \AR{H}$; when the labeling and the induced weight are clear, we denote the weighted sum by $\AR{G} \oplus \AR{H}$. 
\end{definition}
  Our research is inspired by a special case of the following theorem:
\begin{theorem}\cite{AAetc} \label{thmWgt1}  Let $\overrightarrow{G}$ be a DDMOG on $n$ vertices with DDM labeling $g$ and let $\overrightarrow{H}$ be an oriented graph on $m$ vertices with a bijective labeling $h: V(\AR{H}) \rightarrow \{n+1,\dots, n+m\}$ such that   for every $0 \leq i \leq \displaystyle \max_{v \in V(\AR{H})} \{|\text{wt}_h(v)|\} \leq n$, it holds
\begin{equation}\label{condthm1}
    \sum_{v \in V_h^i(\AR{H})} h(v)=\sum_{v \in V_h^{-i}(\AR{H})} h(v).
\end{equation}   Then, $\AR{G} \oplus_{\text{wt}_h} \AR{H}$ is a DDMOG on $n+m$ vertices. \end{theorem}
The authors of \cite{AAetc} offered only one generalized example of a new DDMOG constructed via a weighted sum;  as a corollary of their Theorem~\ref{thmWgt1}, they offered the following result: 
\begin{corollary}\label{corKC4}
Given  $k \in \mathbb{Z}^+$, and given any DDMOG $\overrightarrow{G_0}$ with  $\vert V(G_0)\vert = n\geq 2k$,   there exists a labeling  $h$ on    $k\overrightarrow{C_4}$ such that $\overrightarrow{G_0} \oplus_{{wt} _h} k\overrightarrow{C_4}$ is a DDMOG. \end{corollary}
In this paper we explore generalizations of Corollary~\ref{corKC4}. Given any $L\geq 3$, we show that four copies of an oriented $L-$cycle can be incorporated into a new DDMOG (Section~\ref{sectionthms}). 
We specify the choices  in the orientation and the labeling of the $L-$cycles that deliver new  DDMOGs, which   are significantly different from the types of DDMOGs   initially presented in \cite{AAetc}.
%
%
\subsection{New notions and initial observations}

    Let $\ar{G}$ and $ \ar{H}$ be two separate graphs. We say that $ \ar{H}$ is {\it attached} to $ \ar{G}$ when we introduce one or more oriented edge(s) between some vertex (or vertices) in $V(G)$ and some  vertex (or vertices) in $V(H)$. 
    Clearly, there are many different ways to attach one graph to another (and produce a new graph with the process); we explore those ways that produce a new graph with DDM labeling. 
    Any vertex $u$ of graph $\ar{G}$ that has multiple  vertices of $\ar{H}$ attached to it, is called a {\it node}.  
    
     An {\it ornament} $\overrightarrow{H}$ is an oriented graph that can be attached to a DDMOG $\overrightarrow{G}$ in such a way that the resulting graph has DDM labeling, while preserving the initial  DDM labeling of $\overrightarrow{G}$. %
A generalized  example of an ornament is an oriented graph $\overrightarrow{H}$ such that Theorem \ref{thmWgt1} holds for any DDMOG $\ar{G}$ with $n$ vertices, for a suitably selected value of $n$.
If $s$ vertices in a DDMOG ${\ar{G}}$ are to be attached to multiple vertices in the ornament ${\ar{H}}$, then $\ar{H}$ is called a {\it $s-$nodes ornament}.   
An example of a $2-$nodes ornament formed by four copies of $\ar{C_3}$ is presented in Subfigure~\ref{fig:b}; the new DDMOG  (Subfigure~\ref{fig:c}) is created via a weighted sum with an oriented wheel (Subfigure~\ref{fig:a}). The theorems presented in Section~\ref{sectionthms} are offering constructions of $s-$nodes ornaments, for
various values of $s$. 
In this paper we only study ornaments formed by cycles -  we call such ornaments {\it cycle ornaments} or c-ornaments if needed to distinguish them from ornaments formed in other ways. An example of a graph that is an ornament but not a c-ornament is the windmill of two  wheels $\ar{W_4}^{(2)}$  introduced in \cite{AAetc}. While $\ar{W_4}^{(2)}$ is a DDMOG itself,  shifting the DDM labeling by $1$,   turns  $\ar{W_4}^{(2)}$ into a $1-$node ornament; then, $\ar{K_1}\oplus_{wt} \ar{W_4}^{(2)}$ is a DDMOG (see Figure 11 in \cite{AAetc}). 

\begin{remark}\label{rem1}
 Consider an oriented $L-$cycle $\ar{H}$ with vertices $v_1, v_2, \hdots, v_L$ and edges $v_1v_2$, $v_2v_3$, $\hdots$, $v_{L-1}v_L$, and $ v_Lv_{L+1}\equiv v_Lv_1$. 
  Given any bijective labeling   $h: V(\ar{H}) \rightarrow \{n+1, n+2, \hdots, n+ L\}$ for any $n\in \mathbb{Z}_0^+$, at each vertex   $v_i$, we have
  \begin{equation}\label{ineqateachv}    
        \text{wt}_h(v_i)=h(v_{i-1})-h(v_{i+1})\neq 0. \end{equation}   This is the reason why no $L-$cycle can have DDM labeling. However the total sum of all the weights at the vertices, which we label by $\text{wt}_h(\ar{H})$, is always zero:
   \begin{equation*}     \text{wt}_h(\ar{H}) =\sum_i \text{wt}_h(v_i)= h(v_1) - h( v_2) + 
   \hdots + h(v_{n-1}) - h(v_n) + h(v_n) - h(v_1) = 0.     \end{equation*}
    \end{remark}
     In this paper we only study c-ornaments of type $4\ar{C_L}$, but we note that the following question, a generalization based on Remark~\ref{rem1}, is not yet answered: {\it  If $\ar{H}$ is an oriented graph paired with a labeling function $h$ such that         $\text{wt}_h(\ar{H})=0$, does there exists an integer $k>0$ such that $k$ copies of $\ar{H}$  form an ornament?}  The $1-$node ornament $\ar{W_4}^{(2)}$ we mentioned earlier satisfies this condition, which is an indication that the answer to this question is positive under some constraints.

  \subsection{First examples of $1-$node and  $2-$nodes ornaments}
  \label{squareO} 
  One  way to label  an oriented $4-$cycle which will result in an ornament was utilized in \cite{AAetc} to generate a $1-$node ornament $\ar{C_4}$ (in general, $k\ar{C_4}$, $k\geq 1$).  
 We point  here to a way to label $\ar{H}=4\ar{C_4}$ that  will result in a $2-$nodes c-ornament.  
This labeling,  as displayed in Figure~\ref{fig:4C4},  creates a new DDMOG by attaching $\ar{H}$ to two nodes in an already established DDMOG (see Figure \ref{fig:4partsOrn}). We note that by Theorem~\ref{thmWgt1}, it is sufficient to show that equality~\eqref{condthm1} is satisfied for $i \in \{2, 6\}$;
these equalities are easily verified.  We will show (Theorem~\ref{thms}) that  $4\ar{C_L}$ is a $2-$nodes ornament for any $L\ge3$, but the labeling scheme for every $L\neq 4$ is different than the one used in Figure \ref{fig:4partsOrn}. 

\begin{figure}
\centering
\begin{tikzpicture}
[scale=.85, node distance=1.4cm, base node/.style={circle,draw,minimum size=25pt}]

\node[base node, very thick,  label={[text=red] below:+2}] (7) at (-12, 4) {$v_1$};
\node[base node, very thick, label={[text=red] above: -2}] (13) at (-10, 6) {$v_7$};

\node[base node, very thick, label={[text=red] below:-6}] (9) at (-10, 4) {$v_3$};
\node[base node, very thick,  label={[text=red] above:+6}] (11) at (-12, 6) {$v_5$};
\path[->, very thick] (7) edge (9);
\path[->, very thick] (9) edge (13);
\path[->, very thick] (13) edge (11);
\path[->, very thick] (11) edge (7);

\node[base node, very thick,  label={[text=red] below:-6}] (10) at (-6, 4) {$v_4$};
\node[base node, very thick, label={[text=red] above:+6}] (12) at (-8, 6) {$v_6$};

\node[base node, very thick,  label={[text=red] below:+2}] (8) at (-8, 4) {$v_2$};
\node[base node, very thick, label={[text=red] above:-2}] (14) at (-6, 6) {$v_8$};
\path[->, very thick] (8) edge (10);
\path[->, very thick] (10) edge (14);
\path[->, very thick] (14) edge (12);
\path[->, very thick] (12) edge (8);

\node[base node, very thick,  label={[text=red] below:-2}] (15) at (-4, 4) {$v_9$};
\node[base node, very thick, label={[text=red] above:+2}] (21) at (-2, 6) {$v_{15}$};

\node[base node, very thick,  label={[text=red] below:+6}] (17) at (-2, 4) { $v_{11}$};
\node[base node, very thick, label={[text=red] above:-6}] (19) at (-4, 6) { $v_{13}$};

\path[->, very thick] (17) edge (15);
\path[->, very thick] (15) edge (19);
\path[->, very thick] (19) edge (21);
\path[->, very thick] (21) edge (17);

\node[base node, very thick,  label={[text=red] below:-2}] (16) at (0, 4) { $v_{10}$};
\node[base node, very thick,  label={[text=red] below:+6}] (18) at (2, 4) { $v_{12}$};

\node[base node, very thick, label={[text=red] above:+2}] (22) at (2, 6) { $v_{16}$};
\node[base node, very thick, label={[text=red] above:-6}] (20) at (0, 6) { $v_{14}$};

\path[->, very thick] (18) edge (16);
\path[->, very thick] (16) edge (20);
\path[->, very thick] (20) edge (22);
\path[->, very thick] (22) edge (18);

\end{tikzpicture}
\caption{ Labeling of $ 4\ar{C_4}$ as a $2-$nodes ornament, that can be attached to any DDMOG with $n\geq6$ vertices. The labeling is defined as $h(v_i)=n+i$, $1\leq i \leq 16$, for a fixed $n\geq 6$. The weight values induced by the labeling are displayed in red. See Figure~\ref{fig:4partsOrn} ($n=6$) for a DDMOG  that utilizes this labeling.} 
\label{fig:4C4}
\end{figure}
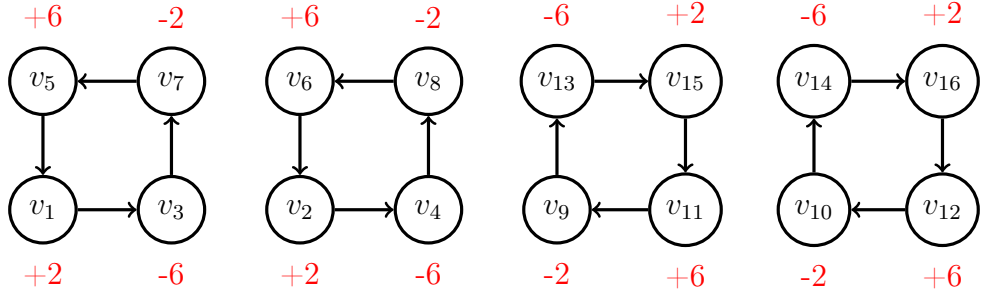

\subsection{Outline}
So far we have presented relevant known results, and we have introduced the notion of an ornament; we also offered an initial example of a DDMOG created via a weighted sum with a $2-$nodes ornament. 
 In Section~\ref{sectionthms}, we provide several results on the existence of  $s-$nodes ornaments, for various values of $s$, where the underlying graph of each ornament is  $4{C_L}$.
 Section~\ref{proofssection} contains the proofs of these results.     In Section~\ref{summary} we summarize the content of the paper and we reflect on additional open questions.  Throughout the paper we make references to  relevant figures, which are located as close as possible to the place of mention.
 
\begin{figure}
\centering
\begin{tikzpicture}
[scale=.85, node distance=1.4cm, base node/.style={circle,draw,minimum size=23pt}]

\node[base node, very thick] (2) at (-3, 2) {2};
\node[base node, very thick] (3) at (-1, 2) {3};
\node[base node, very thick] (4) at (1, 2) {4};
\node[base node, very thick] (5) at (3, 2) {5};

\node[base node, very thick] (6) at (0, 0) {6};
\node[base node, very thick] (1) at (0, 4) {1};
\node[base node, very thick] (7) at (-8, 4) {7};
\node[base node, very thick] (13) at (-6, 6) {13};

\node[base node, very thick] (9) at (-6, 4) {9};
\node[base node, very thick] (11) at (-8, 6) {11};
\path[->, very thick] (7) edge (9);
\path[->, very thick] (9) edge (13);
\path[->, very thick] (13) edge (11);
\path[->, very thick] (11) edge (7);

\path[->, very thick, blue] (7) edge[bend right=13] (2);
\path[->, very thick, blue] (2) edge[bend right=13] (13);
\path[->, very thick, red] (11) edge[bend right=27] (6);
\path[->, very thick, red] (6) edge[bend left=16] (9);

\node[base node, very thick] (10) at (-6, -4) {10};
\node[base node, very thick] (12) at (-8, -2) {12};

\node[base node, very thick] (8) at (-8, -4) {8};
\node[base node, very thick] (14) at (-6, -2) {14};
\path[->, very thick] (8) edge (10);
\path[->, very thick] (10) edge (14);
\path[->, very thick] (14) edge (12);
\path[->, very thick] (12) edge (8);

\path[->, very thick, blue] (8) edge[bend left=75] (2);
\path[->, very thick, blue] (2) edge[bend right=13] (14);
\path[->, very thick, red] (12) edge[bend left=13] (6);
\path[->, very thick, red] (6) edge[bend right=6] (10);

\node[base node, very thick] (15) at (-12, 4) {15};
\node[base node, very thick] (21) at (-10, 6) {21};

\node[base node, very thick] (17) at (-10, 4) {17};
\node[base node, very thick] (19) at (-12, 6) {19};

\path[->, very thick, red] (6) edge[bend left=36] (19);
\path[->, very thick, red] (17) edge[bend right=27] (6);
\path[->, very thick, blue] (21) edge[bend right=36] (2);
\path[->, very thick, blue] (2) edge[bend left=27] (15);

\path[->, very thick] (17) edge (15);
\path[->, very thick] (15) edge (19);
\path[->, very thick] (19) edge (21);
\path[->, very thick] (21) edge (17);

\node[base node, very thick] (16) at (-2, -4) {16};
\node[base node, very thick] (18) at (0, -4) {18};

\node[base node, very thick] (22) at (0, -2) {22};
\node[base node, very thick] (20) at (-2, -2) {20};

\path[->, very thick, red] (6) edge (20);
\path[->, very thick, red] (18) edge[bend right=27] (6);
\path[->, very thick, blue] (22) edge  (2);
\path[->, very thick, blue] (2) edge[bend right=13] (16);

\path[->, very thick] (18) edge (16);
\path[->, very thick] (16) edge (20);
\path[->, very thick] (20) edge (22);
\path[->, very thick] (22) edge (18);

\path[->, very thick] (2) edge (3);
\path[->, very thick] (5) edge[bend right=30] (3);
\path[->, very thick] (5) edge (4);
\path[->, very thick] (2) edge[bend left=30] (4);

\path[->, very thick] (6) edge (2);
\path[->, very thick] (3) edge (6);
\path[->, very thick] (4) edge (6);
\path[->, very thick] (6) edge (5);
 

 \path[->, very thick] (1) edge[bend right=30] (2);
 \path[->, very thick] (1) edge[bend left=30] (5);
 \path[->, very thick] (3) edge (1);
\path[->, very thick] (4) edge (1);
\end{tikzpicture}
\caption{ {DDM } labeling of $\ar{G_0} \oplus_{\text{wt}} 4\ar{C_4}$. } 
\label{fig:4partsOrn}
\end{figure}
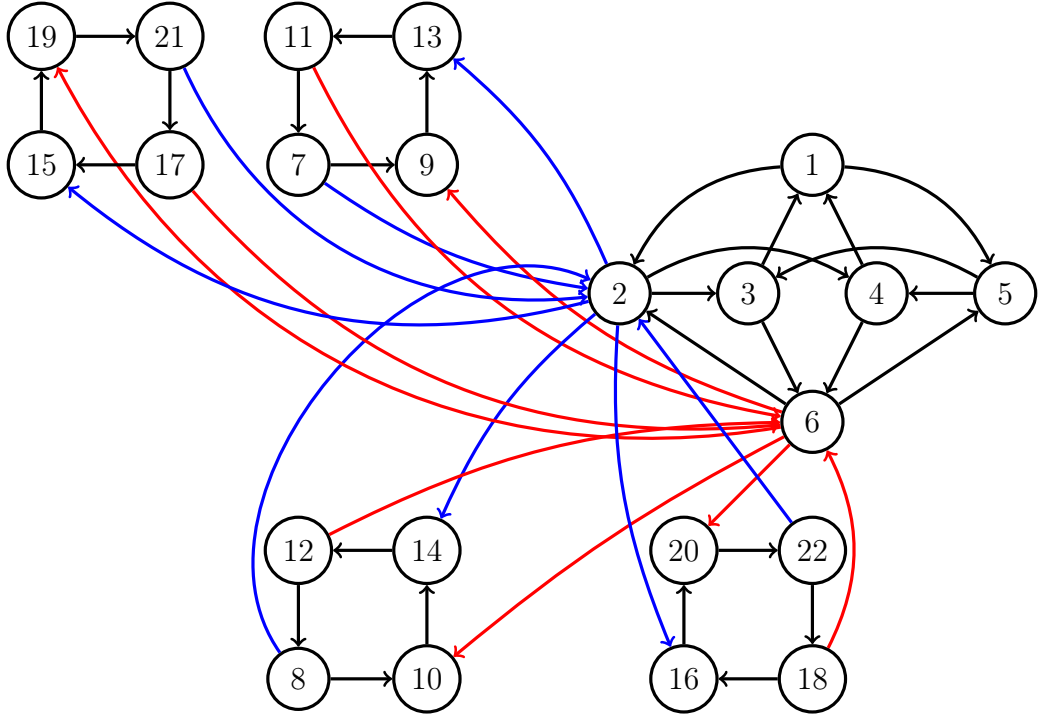
\section{Main results and examples}\label{sectionthms}

It is known \cite{AAetc}  that $k$ copies of an oriented $4-$cycle form  a $1-$node ornament where $k\geq 1$; as shown in subsection~\ref{squareO}, four copies of an oriented  $4-$cycle   also form a $2-$nodes ornament.
We now generalize these initial results and show that $4\ar{C_L}$ (with mild constraints on $L$) is a $s-$ node ornament, with $2\leq s \leq 6$:
\begin{theorem}\label{thms}
    Let $\ar{G_0}$ be a DDMOG with $n$ vertices, and let  $s \in \{2,3, 4,5,6\}$. Let $L$ be an integer such that  $a_s\leq L\le b_s$, where    $a_s$, $b_s$ take values as listed in Table~\ref{s-table}, with an additional constraint   when $s=3$ ($L$ is assumed to be even in this case).
Let   $H=4 C_L$. 
There exists 
     an orientation on  $ {H} $, paired with  a  labeling function $h$, such that 
     \begin{itemize}
         \item[a.] the oriented graph $ \ar{H}$ is a $s-$nodes ornament, where the labels of the nodes are as listed in Table~\ref{s-table}, and 
      \item[b.]  $\displaystyle  \ar{G}:=\ar{G_0}\oplus_{\text{wt}_h} \ar{H}$  is a DDMOG.
      \end{itemize}
\end{theorem}
%
%
\begin{table}[!h]
\begin{center}
\caption{Lower and upper bounds on $L$ for a $s-$nodes ornament $4\ar{C_L}$.
}
\label{s-table}
\begin{tabular}{||c c c c c||} 
 \hline
 $s$ & $a_s$ & $b_s$ & Nodes labels & Constraint\\ [0.5ex] 
 \hline\hline
 2 & 3 & $n+2$ & \hspace{-10mm} $2, \, L-2$ & \\ 
 \hline
 3 & 6 & $n+5$ &  \; $ 4, \, L-5, \, L-3$ & only for $L=2k $\\
 \hline
 4 & 6 & $n+2$ & $1,\, 2, \, 3, \, L-2$ &  \\
 \hline
 5 & 7 & $n+3$ & \;\;\;\;\;\;\;\;\;\; $1,\, 2, \, 3, \, L-3, \, L-2$ &  \\
 \hline
 6 & 8 & $n+3$ & \;\;\; \;\;\;\;\;\;\;\;\;\; $1,\, 2, \, 3, \, 4, \,  L-3, \, L-2$ &   \\ [1ex] 
 \hline
\end{tabular}
\end{center}
\end{table}

We placed the proof of Theorem~\ref{thms} in    Section~\ref{proofssection}. Here we make a  remark on the labeling pattern as utilized in the proof: The labels in each cycle copy  will have a {\it{repetitive, consecutive}} pattern (with translation shifts in the indexing when needed).  
 More specifically, let $\{l(i)\}_{i=1}^L$ be a permutation of the set $\{1, 2, \hdots, L\}$. Given the vertex sets   $  V({H_j})=\{v_{i}^{(j-1)}\}_{i=1}^L$  of the $L-$cycles  $H_j$, $j\in \{1,2,3,4\}$, for the labeling $h$ we  have:
  \begin{equation}\label{repetPatt}
  \text{$h(v_{i}^{(j-1)})=n+l(i)+(j-1)L$ for $1\leq j \leq 4$, $1\leq i \leq L$. }
  \end{equation}
\noindent Observe that when using the labeling pattern \eqref{repetPatt}, $L$ consecutive integers are used in the labeling  within a single cycle. 
Next, we offer a few comments on  each separate case:
 
 $s=2$:    Figure~\ref{4kL} illustrates the labeling of the ornament as used in the proof for any $L\neq 4$. The special case when $L=4$ is already reviewed in subsection~\ref{squareO} (Figure~\ref{fig:4partsOrn}), with a different labeling scheme. 
In fact,   the labeling scheme as used in the proof of Theorem~\ref{thms} fails to generate a $2-$nodes ornament in the case $L=4$, as then also  $L-2=2$.
A detailed overview of an example construction of a $2-$nodes ornament (when $L=3$) and the resulting DDMOG is given in Figure~\ref{fig:c}. Another example of an ornament relevant to Theorem~\ref{thms} (when $s=2$) is placed in  Figure~\ref{fig:triangles2onsingletons}; we conjecture that this DDMOG has the minimal number of edges among all DDMOGs with $14$ vertices.

$s=3$:  In this case the labeling $h$ we use  comes with an extra  constraint: $L$ must be an even integer. 
The three nodes
will have labels $4$, $L-3$ and $L-5$; thus, it must be that $L\geq 6$. Notice that   when $L$ is odd, the construction    will result with a $2-$nodes ornament  - those nodes would be labeled with $4$ and $L-4$.  
 It is possible to create a $3-$nodes ornament when  $L$ is odd,  but instead of a generalized statement that would hold true  for  all odd $L$ we can only  offer an example  of a $3-$nodes ornament when $L=5$ (Figure~\ref{fig:pentasimpleInversion}).

\begin{figure}
\centering
\begin{tikzpicture}
[scale=.9, node distance=1.4cm, base node/.style={circle,draw,minimum size=27pt}]

\node[base node, very thick, label={[text=red]left:+2}](n+1) at (-2.5, 2) {$v_{1}$};
  
\node[base node, very thick, label={[text=red]right:-1}] (n+3) at (-1, 3) {$v_{3}$};
 
\node[base node, very thick, label={[text=red] right:-1}] (n+2) at (0.5, 2){$v_{2}$};

\node[base node, very thick, label={[text=red]right:-3}] (n+4) at (0.5,0) {$v_{4}$};
\node[base node, very thick, label={[text=red] left:+3}] (n+5) at (-2.5,0){$v_{5}$};
 
\path[->, very thick] (n+1) edge (n+3);
\path[->, very thick] (n+3) edge (n+2);
 \path[->, very thick] (n+2) edge (n+4);
 
 \path[->, very thick] (n+4) edge  (n+5);
\path[->, very thick] (n+5) edge (n+1);

\node[base node, very thick, label={[text=red]left:-2}](n+6) at (3.5, 2) {$v_{6}$};
  
\node[base node, very thick, label={[text=red]right:+1}] (n+8) at (5, 3) {$v_{8}$};
 
\node[base node, very thick, label={[text=red]right:+1}] (n+7) at (6.5, 2) {$v_{7}$};

\node[base node, very thick, label={[text=red]right:+3}] (n+9) at (6.5,0) {$v_{9}$};
 
\node[base node, very thick, label={[text=red] left:-3}] (n+10) at (3.5,0) {$v_{10}$};
 
\path[->, very thick] (n+6) edge (n+10);
\path[->, very thick] (n+10) edge (n+9);
 \path[->, very thick] (n+7) edge (n+8);
 
 \path[->, very thick] (n+9) edge  (n+7);
\path[->, very thick] (n+8) edge (n+6);
\node[base node, very thick, label={[text=red]left:+2}](n+16) at (-2.5, -3) {$v_{16}$};
  
\node[base node, very thick, label={[text=red]right:-1}] (n+18) at (-1, -2) {$v_{18}$};
 
\node[base node, very thick, label={[text=red] right:-1}] (n+17) at (0.5, -3) {$v_{17}$};

\node[base node, very thick, label={[text=red]right:-3}] (n+19) at (0.5,-5) {$v_{19}$};
 
\node[base node, very thick, label={[text=red] left:+3}] (n+20) at (-2.5,-5) {$v_{20}$};
 
\path[->, very thick] (n+19) edge (n+20);
\path[->, very thick] (n+20) edge (n+16);
 \path[->, very thick] (n+16) edge (n+18);
 
 \path[->, very thick] (n+18) edge  (n+17);
\path[->, very thick] (n+17) edge (n+19);

\node[base node, very thick, label={[text=red]left:-2}](n+11) at (3.5, -3) {$v_{11}$};
  
\node[base node, very thick, label={[text=red]right:+1}] (n+13) at (5, -2) {$v_{13}$};
 
\node[base node, very thick, label={[text=red]right:+1}] (n+12) at (6.5, -3) {$v_{12}$};

\node[base node, very thick, label={[text=red] right:+3}] (n+14) at (6.5,-5) {$v_{14}$};
 
\node[base node, very thick, label={[text=red]  left:-3}] (n+15) at (3.5,-5) {$v_{15}$};
 
\path[->, very thick] (n+13) edge (n+11);
\path[->, very thick] (n+12) edge (n+13);
 \path[->, very thick] (n+14) edge (n+12);
 
 \path[->, very thick] (n+15) edge  (n+14);
\path[->, very thick] (n+11) edge (n+15);
 
\end{tikzpicture}
\caption{Labeling of a 3-nodes ornament, $4\ar{C_5}$: $h(v_j)=n+j$, $1\leq j \leq 20$, for any $n\geq 3$. The values of $\text{wt}_h$ are marked in red.}
\label{fig:pentasimpleInversion}
\end{figure}
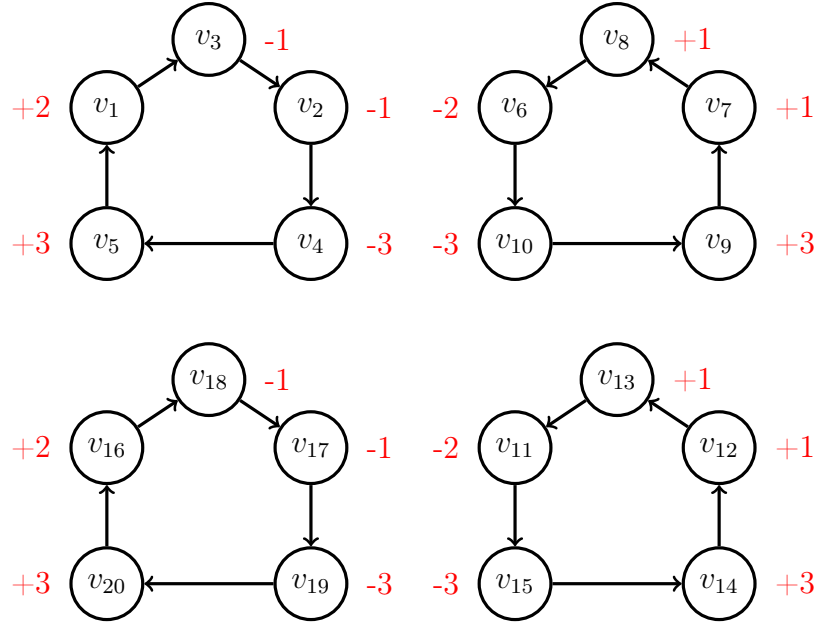

\begin{figure}
\centering
\begin{tikzpicture}
[scale=.9, node distance=1.4cm, base node/.style={circle,draw,minimum size=27pt}]

\node[base node, very thick, label={[text=red]left:+4}](n+2) at (-2.5, 2) {$v_{2}$};
  
\node[base node, very thick, label={[text=red]right:-1}] (n+1) at (-1, 3) {$v_{1}$};
 
\node[base node, very thick, label={[text=red] right:-3}] (n+3) at (0.5, 2) {$v_{3}$};

\node[base node, very thick, label={[text=red]right:-2}] (n+4) at (0.5,0) {$v_{4}$};
 
\node[base node, very thick, label={[text=red] left:+2}] (n+5) at (-2.5,0) {$v_{5}$};
 
\path[->, very thick] (n+1) edge (n+3);
\path[->, very thick] (n+3) edge (n+4);
 \path[->, very thick] (n+4) edge (n+5);
 
 \path[->, very thick] (n+5) edge  (n+2);
\path[->, very thick] (n+2) edge (n+1);


\node[base node, very thick, label={[text=red]right:+1}] (n+6) at (5, 3) {$v_{6}$};
 
\node[base node, very thick, label={[text=red]right:+3}] (n+8) at (6.5, 2) {$v_{8}$};

\node[base node, very thick, label={[text=red]right:+2}] (n+9) at (6.5,0) {$v_{9}$};
 
\node[base node, very thick, label={[text=red] left:-2}] (n+10) at (3.5,0) {$v_{10}$};
\node[base node, very thick, label={[text=red]left:-4}](n+7) at (3.5, 2) {$v_{7}$};
 
\path[->, very thick] (n+6) edge (n+7);
\path[->, very thick] (n+7) edge (n+10);
 \path[->, very thick] (n+10) edge (n+9);
 
 \path[->, very thick] (n+9) edge  (n+8);
\path[->, very thick] (n+8) edge (n+6);
 
\node[base node, very thick, label={[text=red]left:+4}](n+17) at (-2.5, -3) {$v_{17}$};
  
\node[base node, very thick, label={[text=red]right:-1}] (n+16) at (-1, -2)  {$v_{16}$};
 
\node[base node, very thick, label={[text=red] right:-3}] (n+18) at (0.5, -3) {$v_{18}$};

\node[base node, very thick, label={[text=red]right:-2}] (n+19) at (0.5,-5)  {$v_{19}$};
 
\node[base node, very thick, label={[text=red] left:+2}] (n+20) at (-2.5,-5)  {$v_{20}$};
 
\path[->, very thick] (n+16) edge (n+18);
\path[->, very thick] (n+18) edge (n+19);
 \path[->, very thick] (n+19) edge (n+20);
 
 \path[->, very thick] (n+20) edge  (n+17);
\path[->, very thick] (n+17) edge (n+16);

\node[base node, very thick, label={[text=red]left:-4}](n+12) at (3.5, -3)  {$v_{12}$};
  
\node[base node, very thick, label={[text=red]right:+1}] (n+11) at (5, -2) {$v_{11}$};
 
\node[base node, very thick, label={[text=red]right:+3}] (n+13) at (6.5, -3)  {$v_{13}$};

\node[base node, very thick, label={[text=red] right:+2}] (n+14) at (6.5,-5)  {$v_{14}$};
 
\node[base node, very thick, label={[text=red]  left:-2}] (n+15) at (3.5,-5)  {$v_{15}$};
 
\path[->, very thick] (n+13) edge (n+11);
\path[->, very thick] (n+14) edge (n+13);
 \path[->, very thick] (n+15) edge (n+14);
 
 \path[->, very thick] (n+12) edge  (n+15);
\path[->, very thick] (n+11) edge (n+12);

\end{tikzpicture}
\caption{Labeling of a $4-$node ornament, $4\ar{C}_5$:   $h(v_j)=n+j$, $1\leq j \leq 20$, for any $n\geq 4$. The values of $\text{wt}_h$ are marked in red.}
\label{fig:pentaInversionAnother}
\end{figure} 
  
 $s=4$: 
 The four nodes of the ornament   will have labels $1$, $2$, $3$, and $L-2$; thus, it must be that $L\ge 6$. If $L<6$,  then we could not produce four distinct labels for the four nodes with our choice of labeling $h$. However, using another labeling pattern we have found that a $4-$nodes ornament for $L=5$ exists  (see Figure~\ref{fig:pentaInversionAnother}). By Lemma~\ref{minLvsS}, it must be that $L\geq 4$; however, we do not know if  in  the critical case, when $L=4$, there exists an orientation and an associated labeling $h$ so that $4{C_4}$ is a $4-$nodes ornament.

 $s=5$: The labeling function $h$ in this case has a simple increasing pattern, except for one inversion between the second and third vertex in each cycle. The weights induced by $h$ by absolute value are  $1$, $2$, $3$, $L-3$, and $L-2$ if  $n+1\geq L\geq 7$.   Figure~\ref{5nodeweights} depicts the generalized case.

 $s=6$:  Figure~\ref{6nodeAttachGen} illustrates the generalized labeling $h$ of the $6-$nodes ornament  $4\ar{C_L}$, that can be attached to any DDMOG with $n\geq L-3$ vertices; in addition, we require that  $ L \geq 8$. The six nodes must have labels $1, 2, 3, 4,  L-3$ and $ L-2$.  

 \begin{remark}\label{rem24}
     The DDMOG in Figure~\ref{fig:triangles2onsingletons} is created by attaching a $2-$nodes ornament to  $2\ar{K_1}$, a trivial DDMOG. This DDMOG has $14$ vertices and only $24$ edges - in comparison,  the DDMOG $\ar{R_{14}}$ as defined in \cite{AAetc} via vertex coalescence  has significantly more edges. We anticipate that   utilizing the technique of weighted sums with oriented cycles  will produce {\it sparser} DDMOGs (in terms of edges).  
\end{remark}

 \begin{figure}
\centering
\begin{tikzpicture}
[scale=.85, node distance=1.4cm, base node/.style={circle,draw,minimum size=23pt}]
\node[base node, very thick](1) at (-8, 0) {1};
  
\node[base node, very thick ] (2) at (2,3) {2};

\node[base node, very thick](3) at (-4.5, 3) {3};
  
\node[base node, very thick ] (4) at (-2.5, 3) {4};
 
\node[base node, very thick ] (5) at (-3.5, 1.5) {5}; 
\node[base node, very thick ] (6) at (3, 2) {6};
 
\node[base node, very thick ] (7) at (5,2) {7};
\node[base node, very thick ] (8) at (4,0.5) {8};

\node[base node, very thick ](9) at (-5.5, -2) {9};
  
\node[base node, very thick ] (10) at (-3.5, -2) {10};
 
\node[base node, very thick ](11) at (-4.5, -3.5) {11};

\node[base node, very thick ] (12) at (-1, -.5) {12};
\node[base node, very thick ] (13) at (1,-.5) {13};
\node[base node, very thick ](14) at (0,-2) {14};


\path[->, very thick] (3) edge (4);
\path[->, very thick] (4) edge (5);
 \path[->, very thick] (5) edge (3);
 
 \path[->, very thick] (6) edge  (8);

\path[->, very thick] (8) edge (7);
\path[->, very thick] (7) edge (6);
 
\path[->, very thick] (9) edge (11);
\path[->, very thick] (11) edge (10);
 \path[->, very thick] (10) edge (9);
 
 \path[->, very thick] (12) edge  (13);
\path[->, very thick] (13) edge (14);
\path[->, very thick] (14) edge (12);
 
 \path[->, very thick, teal] (3) edge (1);
\path[->, very thick, orange] (2) edge (13); 
\path[->, very thick, teal] (14) edge  (1); 

 \path[->, very thick, teal] (12) edge (1);
\path[->, very thick, orange] (2) edge (4); 
\path[->, very thick, teal] (5) edge  (1); 

 \path[->, very thick, orange] (7) edge (2);
\path[->, very thick, teal] (1) edge  (6); 
\path[->, very thick, teal] (1) edge  (8); 

 \path[->, very thick, orange] (10) edge (2);
\path[->, very thick, teal] (1) edge [bend right=11](11); 
\path[->, very thick, teal] (1) edge   (9); 

\end{tikzpicture}
\caption{DDM labeling of $2\ar{K_1}\oplus_{wt} 4\ar{C_3}$.} 
\label{fig:triangles2onsingletons}
\end{figure}
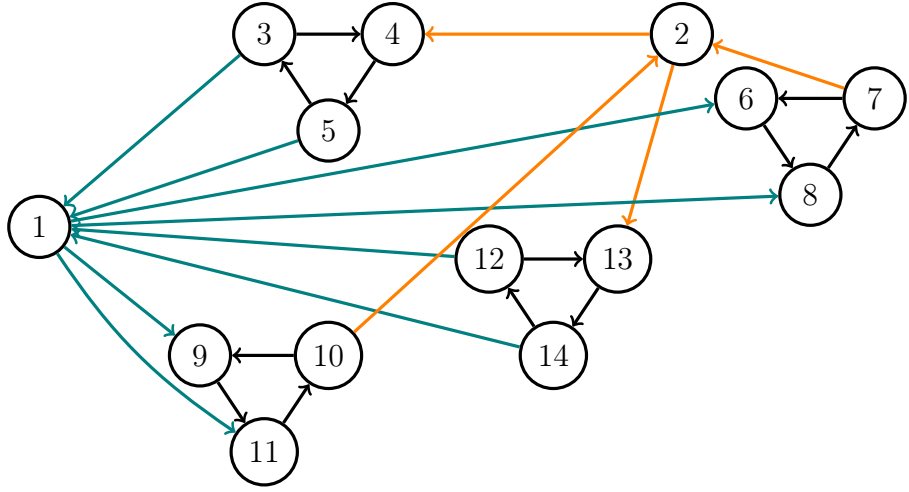

The proof of Theorem~\ref{thms} offers contructions of  $s-$node ornaments, $ 2 \leq s \leq 6$. We anticipate that other
ways to organize the vertex labels will lead to the  creation of $s-$nodes ornaments for values of $s>6$, provided $L$ is large enough; for instance,  we were able to create a $7-$nodes ornament, $7\ar{K_1}\oplus 4\ar{C_9}$ (Figure \ref{7nodeAttach1NewDDMOG}). 
Instead of pursuing formalized ways of generating $s-$nodes ornaments for $s >6$, we prefer to answer the following question: {\it Given a fixed $L\geq 5$,  what is the largest value $s$ can have such that $4\ar{C_L}$  is a $s-$nodes ornament?} We will first show that $s$ cannot be larger than $L$:

  \begin{lemma}\label{minLvsS}
Let $L\geq 3$ and $s\geq 1$. If  $4\ar{C}_L$ can be labeled to form an $s$-nodes ornament, then $s \le L$.
\end{lemma}

\begin{proof} Let $H=4{C_L}$, $L\geq 3$, and let $h: V(H)\rightarrow \{n+1,n+2, \hdots, n+4L\}$  be a bijection, where  $n \in \mathbb{N}^+$.  
  Given  any  DDMOG $\ar{G_0}$ with $n$ vertices, suppose $\ar{G}=\ar{G_0} \oplus_{\text{wt}_h}  \ar{H}$ is a DDMOG on $n+4L$ vertices.  
Let $W_k$ be one of the distinct absolute weight values  of the weight function $\text{wt}_h$;  then  equality~\eqref{condthm1} is satisfied for $i=W_k$. It is easy to see that it must hold $n \geq \max_k  \{ W_k\}$. In addition, because an ornament must work universally for \textit{any} valid DDMOG $\ar{G_0}$,   equality~\eqref{condthm1} must hold (with $i=W_k$) for all such $n \geq \max_k  \{ W_k\}$. Therefore,  the following condition must be satisfied:
\begin{equation*} |V_h^{W_k}(\ar{H})| = |V_h^{-W_k}(\ar{H})| = m_k.\end{equation*} 
This means that there are exactly $m_k$ 'positive' vertices and $m_k$ 'negative' vertices for the weight class $W_k$. 
Since the labeling $h$ is a bijection, all vertex labels in the ornament must be distinct. If $m_k = 1$, there exists exactly one vertex $v \in V_h^{W_k}({H})$ and another vertex $u \in V_h^{-W_k}({H})$, forcing the label sum condition to simplify to $h(v) = h(u)$, which cannot happen as $h$ is a bijection. Therefore, we must have $m_k \ge 2$ for every weight class, meaning  at least $2m_k \ge 4$ vertices are involved. 
Due to a generalization of inequality \eqref{ineqateachv}, no vertex has a weight $\text{wt}_h$ of $0$.  Also, every single vertex in $V({H})$ belongs to exactly one of the $s$ distinct non-zero weight classes. Because each of the $s$ classes requires at least $4$ vertices, any valid $s$-nodes ornament must contain at least $4s$ vertices. 
Then it must hold: $4L \ge 4s$; thus $ s \le L.$
\end{proof}
 

\begin{figure}
\centering
\begin{tikzpicture}
[scale=.94, node distance=1.4cm, base node/.style={circle,draw,minimum size=23pt}]

\node[base node, very thick] (8) at (-7, 2) {8};
\node[base node, very thick] (10) at (-5, 2) {10};
\node[base node, very thick] (12) at (-3, 2) {12};
\node[base node, very thick] (14) at (-1, 2) {14}; 
\node[base node, very thick] (16) at (0.5, 1) {16}; 
\node[base node, very thick](13) at (-1, 0) {13};  
\node[base node, very thick] (15) at (-7, 0) {15};
\node[base node, very thick] (11) at (-5, 0) {11};
\node[base node, very thick] (9) at (-3, 0) {9};

\node[base node, very thick] (17) at (-7, -2) {17};
\node[base node, very thick] (19) at (-5, -2) {19};
\node[base node, very thick] (21) at (-3, -2) {21};
\node[base node, very thick] (23) at (-1, -2) {23}; 
\node[base node, very thick] (25) at (0.5, -3) {25};
\node[base node, very thick](22) at (-1, -4) {22};  
\node[base node, very thick] (24) at (-7, -4) {24};
\node[base node, very thick] (20) at (-5, -4) {20};
\node[base node, very thick] (18) at (-3, -4) {18};

\node[base node, very thick](26) at (-7, -6) {26};
\node[base node, very thick] (28) at (-5, -6) {28};
\node[base node, very thick] (30) at (-3, -6) {30};
\node[base node, very thick] (32) at (-1, -6) {32};
\node[base node, very thick] (34) at (0.5, -7) {34};
\node[base node, very thick](31) at (-1, -8) {31};
\node[base node, very thick] (33) at (-7, -8) {33};
\node[base node, very thick] (29) at (-5, -8) {29};
\node[base node, very thick] (27) at (-3, -8) {27};

\node[base node, very thick](35) at (-7, -10) {35};
\node[base node, very thick] (37) at (-5, -10) {37};
\node[base node, very thick] (39) at (-3, -10) {39};
\node[base node, very thick] (41) at (-1, -10) {41}; 
\node[base node, very thick] (43) at (0.5, -11) {43};
\node[base node, very thick](40) at (-1, -12) {40};  
\node[base node, very thick] (42) at (-7, -12) {42};
\node[base node, very thick] (38) at (-5, -12) {38};
\node[base node, very thick] (36) at (-3, -12) {36};

\node[base node, very thick, ](1) at (3, -5) {1};
\path[->, very thick] (1) edge (34);
\path[->, very thick] (1) edge (25);
\path[->, very thick] (43) edge (1);
\path[->, very thick] (16) edge (1);
\node[base node, very thick, ](7) at (5, -8) {7};
\path[->, very thick] (7) edge (22);
\path[->, very thick] (7) edge (31);
\path[->, very thick] (40) edge[bend right=21]  (7);
\path[->, very thick] (13) edge[bend left=21] (7);

\node[base node, very thick, ](2) at (5, -3) {2};
\path[->, very thick] (9) edge (2);
\path[->, very thick] (2) edge[bend left=27] (18);
\path[->, very thick] (2) edge (27);
\path[->, very thick] (36) edge[bend right=17] (2);

\node[base node, very thick, ](3) at (-9, -7) {3};
\path[->, very thick] (3) edge (33);
\path[->, very thick] (3) edge (24);
\path[->, very thick] (15) edge (3);
\path[->, very thick] (42) edge (3);

\node[base node, very thick, teal](4) at (5, 2) {4};
\node[base node, very thick, teal](4a) at (5, -10) {4'};
\path[->, very thick, teal] (4) edge (14);
\path[->, very thick, teal] (4) edge[bend right=23]  (12);
\path[->, very thick, teal] (4) edge[bend right=27]  (10);
\path[->, very thick,teal] (23) edge (4);
\path[->, very thick,  teal] (19) edge (4);
\path[->, very thick, teal] (21) edge (4);
\path[->, very thick, teal] (4a) edge (41);
\path[->, very thick, teal] (4a) edge[bend right=21]  (39);
\path[->, very thick, teal] (4a) edge[bend right=21]  (37);
\path[->, very thick,teal] (28) edge (4a);
\path[->, very thick,  teal] (30) edge (4a);
\path[->, very thick, teal] (32) edge[bend left=27] (4a);

\node[base node, very thick, ](5) at (-10, -5) {5};
\path[->, very thick] (5) edge (17);
\path[->, very thick] (5) edge (26);
\path[->, very thick] (8) edge (5);
\path[->, very thick] (35) edge[bend left=23] (5);
\node[base node, very thick, ](6) at (-10, -10) {6};
\path[->, very thick] (6) edge (38);
\path[->, very thick] (6) edge[bend right=5] (11);
\path[->, very thick] (29) edge (6);
\path[->, very thick] (20) edge[bend right=-5] (6);

\node[base node, very thick] (8) at (-7, 2) {8};
  
\node[base node, very thick] (10) at (-5, 2) {10};
 \node[base node, very thick] (12) at (-3, 2) {12};

  \node[base node, very thick] (14) at (-1, 2) {14};
 
  \node[base node, very thick] (16) at (0.5, 1) {16};

 \node[base node, very thick](13) at (-1, 0) {13};
  
\node[base node, very thick] (15) at (-7, 0) {15};
 \node[base node, very thick] (11) at (-5, 0) {11};

  \node[base node, very thick] (9) at (-3, 0) {9};

\path[->, very thick] (8) edge (10);
 
\path[->, very thick] (10) edge (12);
 
\path[->, very thick] (12) edge (14);

\path[->, very thick] (14) edge (16);
 
\path[->, very thick] (16) edge (13);

\path[->, very thick] (13) edge (9);

\path[->, very thick] (9) edge (11);

\path[->, very thick] (11) edge (15);
 
\path[->, very thick] (15) edge (8);
 
\node[base node, very thick] (17) at (-7, -2) {17};
  
\node[base node, very thick] (19) at (-5, -2) {19};
 \node[base node, very thick] (21) at (-3, -2) {21};

  \node[base node, very thick] (23) at (-1, -2) {23};
 
  \node[base node, very thick] (25) at (0.5, -3) {25};

 \node[base node, very thick](22) at (-1, -4) {22};
  
\node[base node, very thick] (24) at (-7, -4) {24};
 \node[base node, very thick] (20) at (-5, -4) {20};

  \node[base node, very thick] (18) at (-3, -4) {18};

\path[->, very thick] (17) edge (24);
 
\path[->, very thick] (24) edge (20);
 
\path[->, very thick] (20) edge (18);

\path[->, very thick] (18) edge (22);
 
\path[->, very thick] (22) edge (25);

\path[->, very thick] (25) edge (23);

\path[->, very thick] (23) edge (21);

\path[->, very thick] (21) edge (19);
 
\path[->, very thick] (19) edge (17);

\node[base node, very thick](26) at (-7, -6) {26};
  
\node[base node, very thick] (28) at (-5, -6) {28};
 \node[base node, very thick] (30) at (-3, -6) {30};

  \node[base node, very thick] (32) at (-1, -6) {32};
 
  \node[base node, very thick] (34) at (0.5, -7) {34};

 \node[base node, very thick](31) at (-1, -8) {31};
  
\node[base node, very thick] (33) at (-7, -8) {33};
 \node[base node, very thick] (29) at (-5, -8) {29};

  \node[base node, very thick] (27) at (-3, -8) {27};

\path[->, very thick] (26) edge (33);
 
\path[->, very thick] (29) edge (27);
 
\path[->, very thick] (33) edge (29);

\path[->, very thick] (27) edge (31);
 
\path[->, very thick] (31) edge (34);

\path[->, very thick] (34) edge (32);

\path[->, very thick] (32) edge (30);

\path[->, very thick] (30) edge (28);
 
\path[->, very thick] (28) edge (26);

\node[base node, very thick](35) at (-7, -10) {35};
  
\node[base node, very thick] (37) at (-5, -10) {37};
 \node[base node, very thick] (39) at (-3, -10) {39};

  \node[base node, very thick] (41) at (-1, -10) {41};
 
  \node[base node, very thick] (43) at (0.5, -11) {43};

 \node[base node, very thick](40) at (-1, -12) {40};
  
\node[base node, very thick] (42) at (-7, -12) {42};
 \node[base node, very thick] (38) at (-5, -12) {38};

  \node[base node, very thick] (36) at (-3, -12) {36};

\path[->, very thick] (35) edge (37);
 
\path[->, very thick] (37) edge (39);
 
\path[->, very thick] (39) edge (41);

\path[->, very thick] (41) edge (43);
 
\path[->, very thick] (43) edge (40);

\path[->, very thick] (40) edge (36);

\path[->, very thick] (38) edge (42);

\path[->, very thick] (36) edge (38);
 
\path[->, very thick] (42) edge (35);

\end{tikzpicture}
\caption{DDM labeling of $7\ar{K_1} \oplus 4\ar{C_9}$ (the vertices labeled with $4$ and $4'$ are the same vertex).}
\label{7nodeAttach1NewDDMOG}
\end{figure}
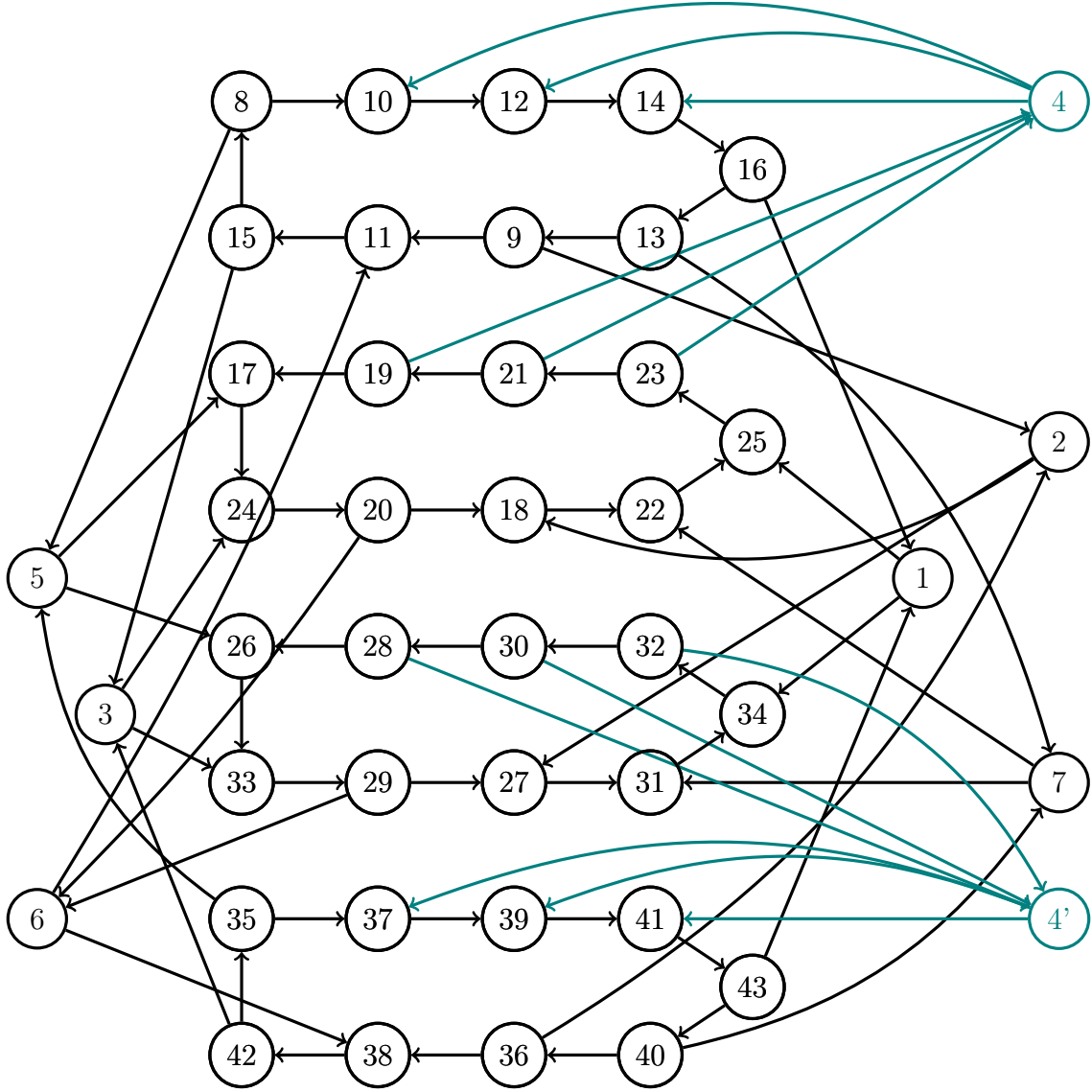


   \begin{theorem}\label{thm:zigzag}
    Let $L\geq 5$, and let $H=4C_L$. There exists an orientation on  $H $, paired with a labeling function $h :  \ar{H} \rightarrow \{1, 2, \hdots, 4L\}$ that has a repetitive, consecutive pattern, such that  $ \ar{H}$ is an $s-$nodes ornament, where $s=L-1$ if $L$ is odd; otherwise $s=(L-2)$.
   \end{theorem} 
\noindent See subsection~\ref{thm:zigzagproof}  for the proof of Theorem~\ref{thm:zigzag}.    Subfigure~\ref{Hep(L-1)} shows a $6-$nodes ornament (when $L=7$) with weights $\pm1, \pm2, \pm3, \pm4, \pm5, \pm6$; that means, the respective nodes have labels $1,2,3,4,5,6$. 
In Subfigure~\ref{Oct(L-2)} we see a $6-$nodes ornament (when $L=8$); the respective nodes  have the following labels: $1, 2, 3, 5, 6, 7$. 

\begin{remark}
  Under the assumption of employing a repetitive, consecutive labeling pattern  \eqref{repetPatt} in  Theorem~\ref{thm:zigzag}, the   number of nodes  can be at most $s = L - 1$,  as the maximum possible absolute difference between any two elements in a set of $L$ consecutive integers is $L - 1$. 
  We were only able to achieve the maximal number of nodes $s=L-1$ in the case when $L$ is odd.
     \end{remark} 
\section{Proofs}~\label{proofssection}
The proof of every result from Section~\ref{sectionthms} makes use of a weighted sum (Definition~\ref{def:whtsumN}) to create a larger, connected  DDMOG by introducing new edges between some vertices (nodes) of an initial DDMOG $\ar{G_0}$ and the vertices of an ornament, composed of four copies of an oriented $L-$cycle.  In every proof we state a specific new labeling function $f$ on the resulting graph such that the weight at every vertex is zero. 
     In addition, while the labels assigned to the vertices of the four cycles  vary from one proof to the other, in all of the proofs we preserve a pairing in terms of {orientation}:   
  $ \ar{H_1}$ and $ \ar{H_4}$  will share an orientation; for instance,   the oriented edges would  be $v_{i}^{(j)} \rightarrow v_{i +1}^{(j)}$ for $j \in \{0,3\}$. Then,  $\ar{H_2}$ and $ \ar{H_3}$  will have the opposite orientation; that is, the oriented edges will be $v_{i +1}^{(j)} \rightarrow v_{i }^{(j)}$ for $j \in \{1, 2\}$. Note that our constructions are not unique; in each proof, if  the orientations of all edges  are flipped, the proof will still hold true.

  Let $\ar{H_1}$ and $\ar{H_4}$ share an   orientation, while  $\ar{H_2}$ and $\ar{H_3}$ share the opposite orientation. 
    Because of the change in the cycle orientation, if for some $i \in \{1,2, \hdots,  L\}$ the initial weight $\text{wt}_h$ at vertices $v_{i}^{(0)} \in V(H_1)$ and $v_{i}^{(3)} \in V(H_4)$ is $q>0$, then at its copies in $ \ar{H_2}$ and $ \ar{H_3}$, the initial weight value  will be $-q$.
        Now, suppose a node $u_{q }$ in the initial DDMOG $\ar{G_0}$ has label $f_0(u_{q})=q$. Introducing new edges from (or to)  $u_{q}$ as needed will change the weight $\text{wt}_f$ at the respective vertex  in $V(H)$ to zero. With that, the weight at  $u_{q}$  remains zero because the additional incoming and outgoing flux due to the new edges to or from $\{v_{i}^{(j)}\}_{j=0}^3$ will balance each other out:
      $(n+l(i)) + (n+l(i)+3L) =  (n+l(i)+L) + (n+l(i)+2L)$. Thus, the weight at every vertex of this new graph is zero, i.e. this  new graph is a DDMOG.
\begin{proof}[Proof of Theorem~\ref{thms}] 

 We present the proof of the theorem for each value of $s \in \{2,3,4,5,6\}$. In each case we work with a DDMOG $\ar{G_0}$, which is an oriented graph with $n $ vertices (where $n$ has some constrains), paired with a DDM labeling bijection $f_0:V(G_0)\rightarrow \{1,...,n\}$. We  work with  four copies  of a $L-$cycle $H_j $, $1\leq j \leq 4$, whose vertices will need to be labeled differently in each case.  
 
 {\texttt{Case s=2:}} 
 Let $L\in \{3,5,6,7,\hdots\}$, and let $n\ge L-2$. Let 
    $H_j=C_L$, $1\leq j\leq 4$, with vertex sets $V(H_{j+1}) = \{ v_i^{(j-1)}\}_{i=1}^L$.
       We define an orientation on $H=\cup_{j=1}^4 H_j$ by specifying the  edges sets of  the oriented cycles $\ar{H_i}$ as follows:  for $j \in \{0, 3\}$, let    $E(\ar{H}_{j+1})=\{ v_i ^{(j)} v_{i+1}^{(j)}  \}_{i=1}^{L-1} \cup \{ v_L ^{(j)}v_{1}^{(j)}\}$, and for $j \in \{1,2\}$, let 
 
  \noindent  $E(\ar{H}_{j+1})=\{ v_{i+1}^{(j)}   v_{i}^{(j)}   \}_{i=1}^{L-1} \cup \{ v_{1}^{(j)}  v_{L}^{(j)} \}$. 
    
    
Let    $h: V(H) \rightarrow \{n+1, \hdots, n+4L\}$ be the bijection defined by  $h(v_i^{(j)})=n+jL+i$,  where  $1\leq i\leq L$, $0 \leq j \leq 3$.  
   Observe that    $\text{wt}_h(v)$   has values $\pm2$ or $\pm(L-2)$ for any $v$ in $V(H)$ (Figure~\ref{4kL} shows the weight values    for any $L\neq 4$).

   Let $f:V({G_0}) \cup V({H})\rightarrow\{1,...,n,...,n+4L\}$ be defined as $f|_{V(G_0)} \equiv f_0$, $f|_{V(H)} \equiv h$;
    %
  it is easy to see that $f$ is a bijection. 
   Now, let $u_2, \, u_{L-2} \in V(G_0)$ be such that $f_0(u_2)=2$ and $f_0(u_{L-2})=L-2$. 
We introduce new oriented edges to or from $u_p$, $p \in \{2, L-2\}$, and collect them  in set $E'$   as follows: 
  $vu_p  \in E'$ if $\text{wt}_h(v) = p$; but  
{if} $\text{wt}_h(v) = -p$,  {then} $u_pv \in E'.$
  With that, we have defined  $E(\ar{G_0}\oplus_{\text{wt}_h} \ar{H}) = E(G_0) \cup  E(H) \cup E'$. 
    Note that due to the  newly added set $E'$, the weight $\text{wt}_f$ of every $v \in V ( {H})$ is $0$.
For any vertex $u \in V({G_0}) \setminus \{u_2, u_{L-2}\}$, we have $\text{wt}_f(u) = \text{wt}_{f_0}(u)=0$. We now  confirm that the weight $\text{wt}_f$ at the vertices $u_2$ and $u_{L-2}$  is zero:

\begin{eqnarray*}
                    \text{wt}_f(u_2)&=& \text{wt}_{f_0}(u_2)+\sum_{x=n+2}^{n+L-1}x-\sum_{x=n+L+2}^{n+2L-1}x-\sum_{x=n+2L+2}^{n+3L-1}x+\sum_{x=n+3L+2}^{n+4L-1}x \\
        &=& 0+ 0.5((L-2)(n+L-1+n+2))\\ & &-0.5((L-2)(n+2L-1+n+L+2))\\& & -0.5(L-2)(n+3L-1+n+2L+2))\\
        & & +0.5((L-2)(n+4L-1+n+3L+2))\\
        &=&0.5{(8L^2+4Ln-14L-8n-4-(8L^2+4Ln-14L-8n-4))}
        =0, \\ \text{wt}_f(u_{L-2}) &=&   (n+1)+(n+L)+(n+3L+1)+(n+4L) \\
      & & -[(n+L+1)+(n+2L)+(n+2L+1)+(n+3L)] \\
      &=& 0+4n+8L+2-(4n+8L+2)=0.
      \end{eqnarray*}
Therefore, for every vertex $v$ in the new graph $\ar{G} = \ar{G_0}\oplus_{\text{wt}_h}  \ar{H}$ we have $\text{wt}_f(v)=0$. By definition,    $\ar{G}$ is DDMOG, and $\ar{H}=\cup_{j=1}^4\ar{H_j}$ is a $2-$nodes ornament.

{\texttt{Case s=3:}} 
    Let $L\geq 6$ be an even integer, and let $u_p\in E(G_0)$, $p \in \{4, L-5, L-3\}$ be such that $f_0(u_p)=p$. 
     Let   $V({H_{j+1}})=\{v_{i}^{(j)} \}_{i=1}^L$, $0\leq j \leq 3$; we define an orientation on $H=\cup_{i=j}^4 H_j$ by choosing the oriented edges sets  in the following way:  
     
    \noindent   $E(\ar{H_j}) =\{v_{i}^{(j)}v_{i+1}^{(j)} \}_{i=1}^{L-1} \cup \{v_{L}^{(j)}v_{1}^{(j)}\}$, for $j \in \{0, 3\}$, and   
     
 \noindent   $E(\ar{H_i}) =\{v_{i+1}^{(j)} v_{i}^{(j)} \}_{i=1}^{L-1} \cup    \{v_{1}^{(j)}v_{L}^{(j)}\}$, for $j \in \{1, 2\}$.

 \noindent In addition,   we assign labels to the vertices in $V({H})$ as follows:
     \begin{equation}\label{Hoddsthenevens}
h(v_{i}^{(j)}) =
\begin{cases}
n+ jL+ 2i-1, & \text{if } 1\leq i \leq L/2, \;\; 0\leq j\leq 3; \\\\
n+jL+2i, & \text{if } L/2 < i \leq L,  \;\; 0\leq j \leq 3.
\end{cases}
\end{equation}
\begin{center}
\begin{tikzpicture} [scale=0.94, node distance=1.4cm, base node/.style={circle,draw,minimum size=36pt}]
 
\node[base node, very thick, label=  {[text=red]left:$L - 2$}] (1) at (-9, -1.5) {$v_{1}^{(0)}$};
\node[base node, very thick, label=  {[text=red]210:-2}] (2) at (-6.25, -1) {$v_{2}^{(0)}$};
\node[base node, very thick, label=  {[text=red]210:-2}] (3) at (-3.5, -1) {$v_{3}^{(0)}$};
\node[base node, very thick, label=  {[text=red]right:-2}] (4) at (-0.75, -1.5) {$v_{4}^{(0)}$};

\node[base node, very thick, label=  {[text=red]left:$L - 2$}] (5) at (-9, -3.5) {$v_{L}^{(0)}$};
\node[base node, very thick, label=  {[text=red]150:-2}] (6) at (-6.25, -4) {$v_{_{L-1}}^{(0)}$};
\node[base node, very thick, label=  {[text=red]135:-2}] (7) at (-3.5, -4) {$v_{_{L-2}}^{(0)}$};
\node[base node, very thick, label=  {[text=red]right:-2}] (8) at (-0.75, -3.5) {$v_{_{L-3}}^{(0)}$};

\node[draw=none, very thick] (9) at (0.75,-2.5) {$\vdots$};

\path[->, very thick] (1) edge (2);
\path[->, very thick] (8) edge (7);
\path[->, very thick] (2) edge (3);
\path[->, very thick] (7) edge (6);
\path[->, very thick] (3) edge (4);
\path[->, very thick] (6) edge (5);
\path[->, very thick] (5) edge (1);
\path[->, very thick] (4) edge (9);
\path[->, very thick] (9) edge (8);


\node[base node, very thick, label=  {[text=red]left:$-(L - 2)$}] (10) at (-9, -6.5) {$v_1^{(1)}$};
\node[base node, very thick, label=  {[text=red]240:-2}] (11) at (-6.25, -6) {$v_2^{(1)}$};
\node[base node, very thick, label=  {[text=red]240:-2}] (12) at (-3.5, -6) {$v_3^{(1)}$};
\node[base node, very thick, label=  {[text=red]right:-2}] (13) at (-0.75, -6.5) {$v_4^{(1)}$};

\node[base node, very thick, label=  {[text=red]left:$-(L - 2)$}] (14) at (-9, -8.5) {$v_L^{(1)}$};
\node[base node, very thick, label=  {[text=red]150:-2}] (15) at (-6.25, -9) {$v_{_{L-1}}^{(1)}$};
\node[base node, very thick, label=  {[text=red]135:-2}] (16) at (-3.5, -9) {$v_{_{L-2}}^{(1)}$};
\node[base node, very thick, label=  {[text=red]right:-2}] (17) at (-0.75, -8.5) {$v_{_{L-3}}^{(1)}$};

\node[draw=none, very thick] (18) at (0.75,-7.5) {$\vdots$};

\path[<-, very thick] (10) edge (11);
\path[<-, very thick] (17) edge (16);
\path[<-, very thick] (11) edge (12);
\path[<-, very thick] (16) edge (15);
\path[<-, very thick] (12) edge (13);
\path[<-, very thick] (15) edge (14);
\path[<-, very thick] (14) edge (10);
\path[<-, very thick] (13) edge (18);
\path[<-, very thick] (18) edge (17);


\node[base node, very thick, label=  {[text=red]left:$-(L - 2)$}] (19) at (-9, -11.25) {$v_1^{(2)}$};
\node[base node, very thick, label=  {[text=red]240:-2}] (20) at (-6.25, -11) {$v_2^{(2)}$};
\node[base node, very thick, label=  {[text=red]240:-2}] (21) at (-3.5, -11) {$v_3^{(2)}$};
\node[base node, very thick, label=  {[text=red]right:-2}] (22) at (-0.75, -11.5) {$v_4^{(2)}$};

\node[base node, very thick, label=  {[text=red]left:$-(L - 2)$}] (23) at (-9, -13.75) {$v_{_{L}}^{(2)}$};
\node[base node, very thick, label=  {[text=red]150:-2}] (24) at (-6.25, -14) {$v_{_{L-1}}^{(2)}$};
\node[base node, very thick, label=  {[text=red]135:-2}] (25) at (-3.5, -14) {$v_{_{L-2}}^{(2)}$};
\node[base node, very thick, label=  {[text=red]right:-2}] (26) at (-0.75, -13.5) {$v_{_{L-3}}^{(2)}$};

\node[draw=none, very thick] (27) at (0.75,-12.5) {$\vdots$};

\path[<-, very thick] (19) edge (20);
\path[<-, very thick] (26) edge (25);
\path[<-, very thick] (20) edge (21);
\path[<-, very thick] (25) edge (24);
\path[<-, very thick] (21) edge (22);
\path[<-, very thick] (24) edge (23);
\path[<-, very thick] (23) edge (19);
\path[<-, very thick] (22) edge (27);
\path[<-, very thick] (27) edge (26);


\node[base node, very thick, label=  {[text=red]left:$L - 2$}] (28) at (-9, -16.25) {$v_1^{(3)}$};
\node[base node, very thick, label=  {[text=red]240:-2}] (29) at (-6.25, -16) {$v_2^{(3)}$};
\node[base node, very thick, label=  {[text=red]240:-2}] (30) at (-3.5, -16) {$v_3^{(3)}$};
\node[base node, very thick, label=  {[text=red]right:-2}] (31) at (-0.75, -16.5) {$v_4^{(3)}$}; 

\node[base node, very thick, label=  {[text=red]left:$L - 2$}] (32) at (-9, -18.75) {$v_{_{L}}^{(3)}$};
\node[base node, very thick, label=  {[text=red]150:-2}] (33) at (-6.25, -19) {$v_{_{L-1}}^{(3)}$};
\node[base node, very thick, label=  {[text=red]135:-2}] (34) at (-3.5, -19) {$v_{_{L-2}}^{(3)}$};
\node[base node, very thick, label=  {[text=red]right:-2}] (35) at (-0.75, -18.5) {$v_{_{L-3}}^{(3)}$};

\node[draw=none, very thick] (36) at (0.75,-17.5) {$\vdots$};

\path[->, very thick] (28) edge (29);
\path[->, very thick] (35) edge (34);
\path[->, very thick] (29) edge (30);
\path[->, very thick] (34) edge (33);
\path[->, very thick] (30) edge (31);
\path[->, very thick] (33) edge (32);
\path[->, very thick] (32) edge (28);
\path[->, very thick] (31) edge (36);
\path[->, very thick] (36) edge (35);

\end{tikzpicture}

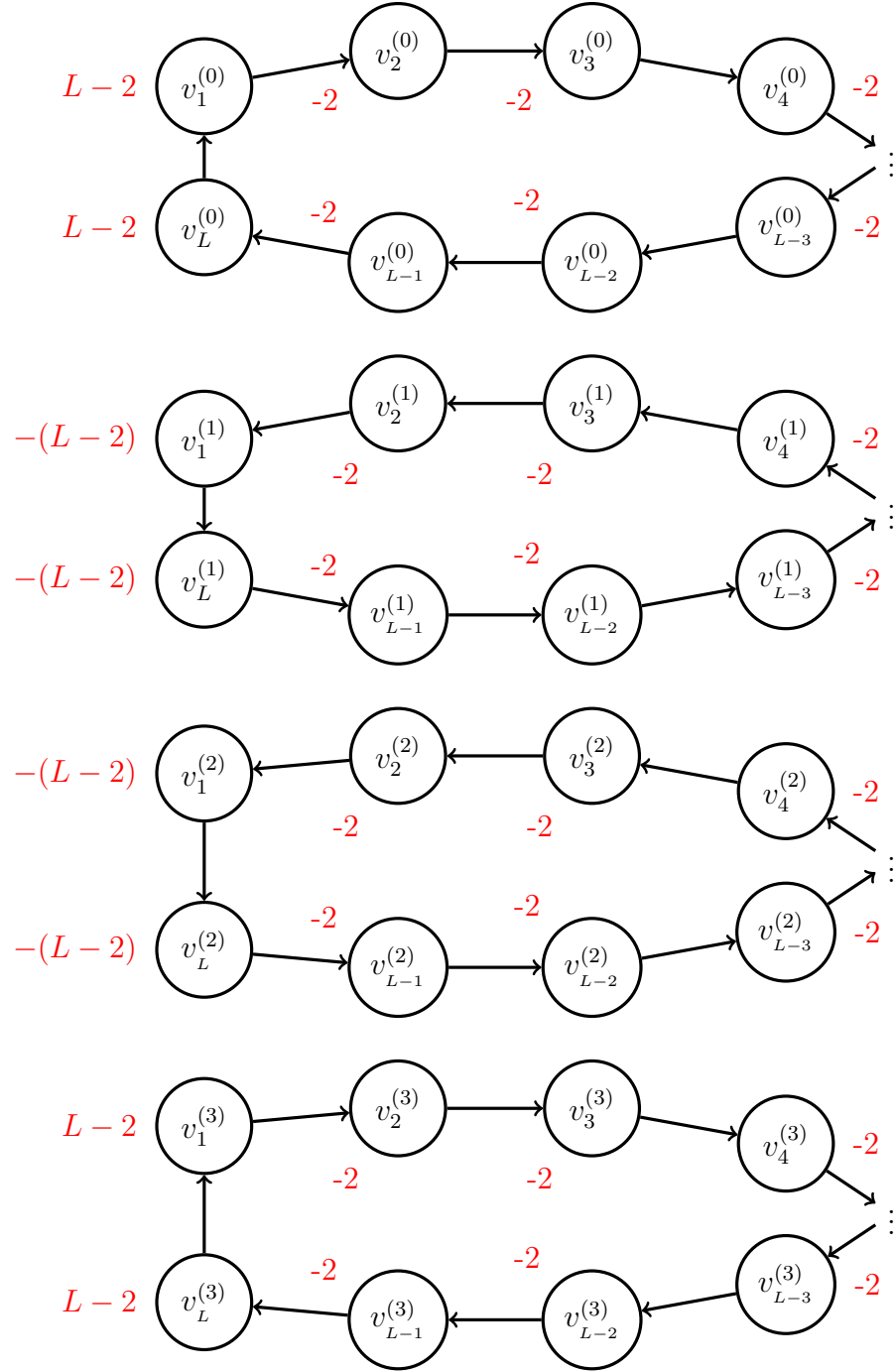
\captionof{figure}{The labeling of $4\ar{C_L}$ used in the proof of   Theorem \ref{thms} for the case $s=2$, $L\neq 4$: $h(v_i ^{(j)})=n+jL+i$, where $1\leq i \leq L$, $0\leq j \leq 3$. The values of $\text{wt}_h$ are shown in red, indicating the nodes satisfy $f_0(u_p)=p$, for $p \in \{2,L-2\}$. }
\label{4kL}
   
 \end{center}

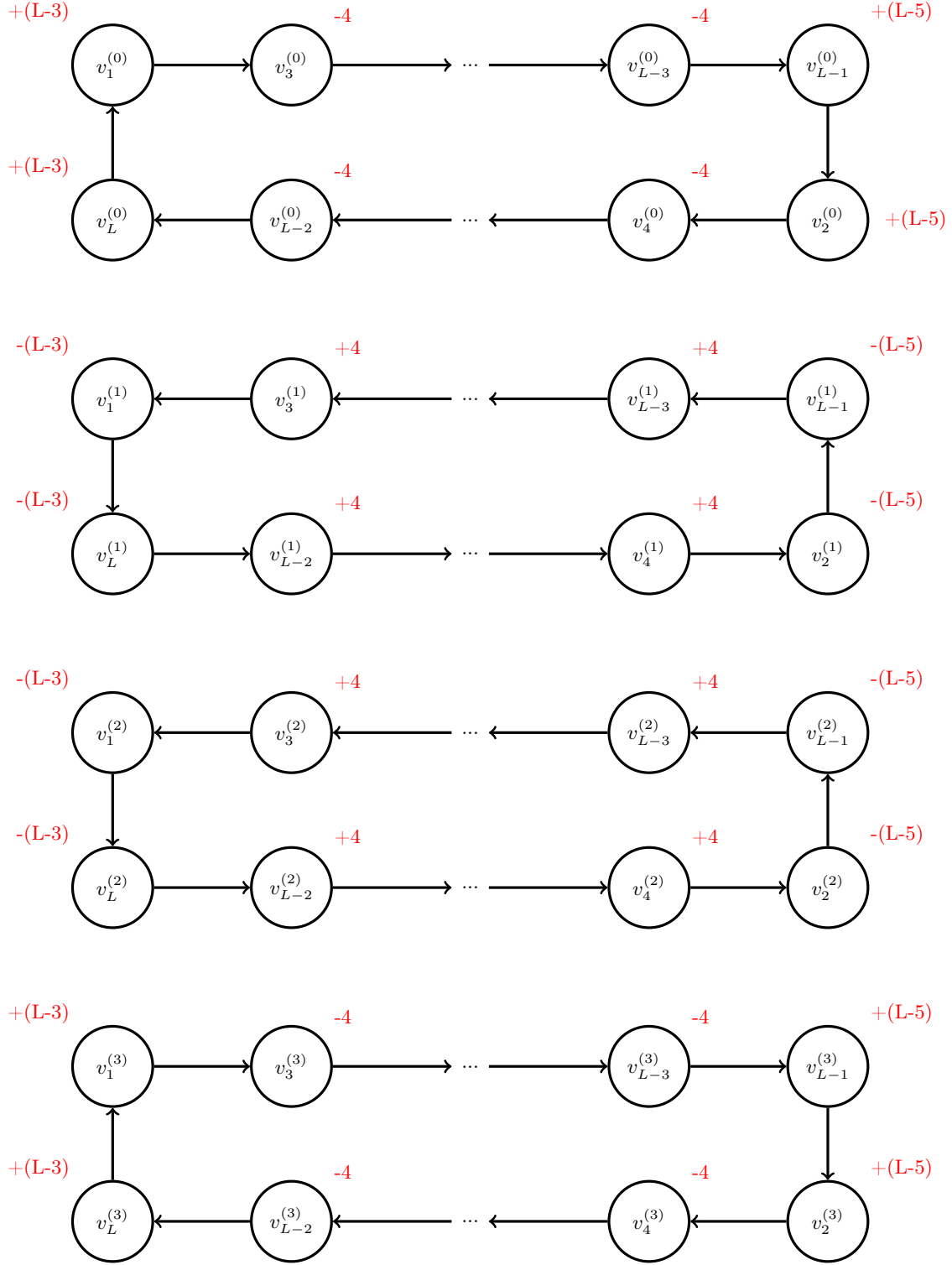
\begin{figure}
\centering
\begin{tikzpicture}
[scale=0.94, font=\scriptsize, node distance=1.4cm, base node/.style={circle,draw,minimum size=36pt}]
 
\node[base node, very thick, label={[text=red]above left:+(L-3)}](n+1) at (-9, 3.6) {$v_1^{(0)}$};
\node[base node, very thick, label={[text=red]above right:-4}](n+3) at (-6, 3.6) {$v_3^{(0)}$};
\node[ very thick ] (dd11) at (-3, 3.6) {...};
\node[base node, very thick, label={[text=red]above right:-4}] (n+L-3) at (0, 3.6) {$v_{L-3}^{(0)}$};
\node[base node, very thick, label={[text=red]above right:+(L-5)}] (n+L-1) at (3, 3.6) {$v_{L-1}^{(0)}$};
\node[base node, very thick, label={[text=red]right:+(L-5)}](n+2) at (3, 1) {$v_{2}^{(0)}$};
\node[base node, very thick, label={[text=red]above right:-4}](n+4) at (0, 1) {$v_{4}^{(0)}$};
\node[ very thick ] (dd12) at (-3, 1) {...};
 
\node[base node, very thick, label={[text=red]above right:-4}] (n+L-2) at (-6, 1) {$v_{L-2}^{(0)}$};

\node[base node, very thick, label={[text=red]above left: +(L-3)}] (n+L) at (-9, 1) {$v_{L}^{(0)}$};

\path[->, very thick] (n+1) edge (n+3);
\path[->, very thick] (n+3)  edge (dd11);
\path[->, very thick] (dd11) edge   (n+L-3);
\path[->, very thick] (n+L-3) edge (n+L-1);
\path[->, very thick] (n+L-1) edge (n+2);
\path[->, very thick] (n+2) edge (n+4); 
\path[->, very thick] (n+4)   edge (dd12);
\path[->, very thick] (dd12) edge (n+L-2);
\path[->, very thick] (n+L-2) edge (n+L);
\path[->, very thick] (n+L) edge (n+1);

\node[base node, very thick, label={[text=red]above left:-(L-3)}](n+L+1) at (-9, -2) {$v_{1}^{(1)}$};

\node[base node, very thick, label={[text=red]above right:+4}](n+L+3) at (-6, -2) { $v_{3}^{(1)}$};

\node[ very thick ](dd21) at (-3, -2) {...}; 
\node[base node, very thick, label={[text=red]above right:+4}] (n+2L-3) at (0, -2) {$v_{L-3}^{(1)}$};

\node[base node, very thick, label={[text=red]above right:-(L-5)}](n+2L-1) at (3, -2) {$v_{L-1}^{(1)}$};

\node[base node, very thick, label={[text=red]above right:-(L-5)}](n+L+2) at (3, -4.6) {$v_{2}^{(1)}$}; 
\node[base node, very thick, label={[text=red]above right:+4}](n+L+4) at (0, -4.6) {$v_{4}^{(1)}$};
 
\node[ very thick ](dd22) at (-3, -4.6) {...};  
 
\node[base node, very thick, label={[text=red]above right:+4}] (n+2L-2) at (-6, -4.6) {$v_{L-2}^{(1)}$};
\node[base node, very thick, label={[text=red]above left:-(L-3)}](n+2L) at (-9, -4.6) {$v_{L}^{(1)}$}; 

\path[->, very thick] (n+L+1) edge (n+2L);
\path[->, very thick] (n+2L) edge (n+2L-2);
\path[->, very thick] (n+2L-2)  edge (dd22);
\path[->, very thick] (dd22) edge   (n+L+4);
\path[->, very thick] (n+L+4) edge (n+L+2);
\path[->, very thick] (n+L+2) edge (n+2L-1);
\path[->, very thick] (n+2L-1) edge (n+2L-3);
\path[->, very thick] (n+2L-3) edge (dd21);
\path[->, very thick] (dd21) edge   (n+L+3);
\path[->, very thick] (n+L+3) edge (n+L+1);
 
\node[base node, very thick, label={[text=red]above left:-(L-3)}](n+2L+1) at (-9, -7.6) {$v_{1}^{(2)}$};
\node[base node, very thick, label={[text=red]above right:+4}](n+2L+3) at (-6, -7.6) {$v_{3}^{(2)}$};
\node[ very thick ](dd31) at (-3, -7.6) {...};
\node[base node, very thick, label={[text=red]above right:+4}] (n+3L-3) at (0, -7.6) { $v_{L-3}^{(2)}$};
\node[base node, very thick, label={[text=red]above right:-(L-5)}](n+3L-1) at (3, -7.6) {$v_{L-1}^{(2)}$};
\node[base node, very thick, label={[text=red]above right:-(L-5)}](n+2L+2) at (3, -10.2) {$v_{2}^{(2)}$};
\node[base node, very thick, label={[text=red]above right:+4}](n+2L+4) at (0, -10.2) {$v_{4}^{(2)}$};
\node[ very thick ](dd32) at (-3, -10.2) {...};
\node[base node, very thick, label={[text=red]above right:+4}] (n+3L-2) at (-6, -10.2) {$v_{L-2}^{(2)}$};
\node[base node, very thick, label={[text=red]above left:-(L-3)}](n+3L) at (-9, -10.2) {$v_{L}^{(2)}$};

\path[->, very thick] (n+2L+1) edge (n+3L);
\path[->, very thick] (n+3L) edge (n+3L-2);
\path[->, very thick] (n+3L-2)   edge (dd32);
\path[->, very thick] (dd32) edge  (n+2L+4);
\path[->, very thick] (n+2L+4) edge (n+2L+2);
\path[->, very thick] (n+2L+2) edge (n+3L-1);
\path[->, very thick] (n+3L-1) edge (n+3L-3);
\path[->, very thick] (n+3L-3)  edge (dd31);
\path[->, very thick] (dd31)   edge (n+2L+3);
\path[->, very thick] (n+2L+3) edge (n+2L+1);

\node[base node, very thick, label={[text=red]above left:+(L-3)}](n+3L+1) at (-9, -13.2) {$v_{1}^{(3)}$};
\node[base node, very thick, label={[text=red]above right:-4}](n+3L+3) at (-6, -13.2) {$v_{3}^{(3)}$};
\node[ very thick ] (dd41) at (-3, -13.2) {...};
\node[base node, very thick, label={[text=red]above right:-4}] (n+4L-3) at (0, -13.2) {$v_{L-3}^{(3)}$};
\node[base node, very thick, label={[text=red]above right:+(L-5)}](n+4L-1) at (3, -13.2) {$v_{L-1}^{(3)}$};
\node[base node, very thick, label={[text=red]above right:+(L-5)}](n+3L+2) at (3, -15.8) {$v_{2}^{(3)}$};
\node[base node, very thick, label={[text=red]above right:-4}](n+3L+4) at (0, -15.8) {$v_{4}^{(3)}$}; 
\node[ very thick ] (dd42) at (-3, -15.8) {...};
\node[base node, very thick, label={[text=red]above right: -4}] (n+4L-2) at (-6,-15.8) {$v_{L-2}^{(3)}$};
\node[base node, very thick, label={[text=red]above left:+(L-3)}](n+4L) at (-9, -15.8) {$v_{L}^{(3)}$};

\path[->, very thick] (n+3L+1) edge (n+3L+3);
\path[->, very thick] (n+3L+3)   edge (dd41);
\path[->, very thick] (dd41) edge   (n+4L-3);
\path[->, very thick] (n+4L-3) edge (n+4L-1);
\path[->, very thick] (n+4L-1) edge (n+3L+2);
\path[->, very thick] (n+3L+2) edge (n+3L+4);
\path[->, very thick] (n+3L+4)  edge (dd42);
\path[->, very thick] (dd42)  edge (n+4L-2);
\path[->, very thick] (n+4L-2) edge (n+4L);
\path[->, very thick] (n+4L) edge (n+3L+1);

\end{tikzpicture}
\caption{$3-$nodes ornament (4$\ar{C_L}$). The labeling used in the proof   of Theorem \ref{thms}, case $s=3$, is specified  in equation \eqref{Hoddsthenevens}. The values of $\text{wt}_h$ are shown in red, indicating the nodes $u_4$, $u_{L-5}$ and $u_{L-3}$ must satisfy $f_0(u_p)=p$, for $p \in \{4, L-5,L-3\}$. }
\label{3nodeweights}
\end{figure}

   The labeling \eqref{Hoddsthenevens} induces a weight $w_h: V(H) \rightarrow \{n+1, \hdots, n+4L\}$ which has values $\pm 4$, $\pm (L-3)$ and $\pm (L-5)$, as shown in Figure~\ref{3nodeweights}.  
It is easy to verify that the weights on the vertices of $ {H_1}$ and  $ {H_4}$ are either $-4$, $L-5$ or $L-3$ when  $L\geq 6$. Similarly, the weights on the vertices of $ {H_2}$ and  ${H_3}$ are either $4$, $-(L-5)$ or $-(L-3)$.

We now define a set of additional edges $E'$:
Let $V_u=\{u_p \, | \, p \in \{4,L-3, L-5 \}\subset V(G_0)$ such that $f_0(u_p)=p$, $ p \in \{4,L-3, L-5 \}$. For any $u_p \in V_u$ and any $v \in V(H)$, we choose:
$vu_p \in E'$ if $\text{wt}_h(v)= p$;  similarly, $u_p v \in E'$ if   $\text{wt}_h(v)= -p$.
  Next, let $\ar{G}:=\ar{G_0} \oplus_{h}   \ar{H}$, where $V({G})=V({G_0}) \cup   V( {H})$ and  $E(\ar{G}):=E(\ar{G_0}) \cup E(\cup_{j=1}^4 \ar{H_j}) \cup E'$. We will now verify by definition that $\ar{G}$ is a DDMOG:
 
 Let $f$ be a labeling on $\ar{G}$, defined by $f (u) =  f_0(u)$ for all $u \in V({G_0})$, and $f(v)=h(v)$ for all $v \in \cup_{j=1}^4 {H_j}$. Obviously, $f: V(\ar{G}) \rightarrow \{1,2, \hdots, n+4L\}$ is a bijection, and for the weight function $\text{wt}_f$ it induces, it holds: $\text{wt}_f(v) = 0$ for all $v \in V(H)$, due to the additionally defined edges  edges packed in $E'$. Also, $\text{wt}_f(u)=\text{wt}_{f_0}(u)=0$ for all $u \in V(G_0)\setminus V_u$. Finally, 
\begin{eqnarray*}
    \text{wt}_f(u_{L-5} ) &=& (n+L-1)+(n+2) - (n+2L-1)-(n+L+2) \\
   & & +(n+4L -1)+(n+3L+2) - (n+3L-1)-(n+2L+2)  =0,
   \\
   \text{wt}_f(u_{L-3} ) &=& (n+1)+(n+L) - (n+L+1)-(n+2L) 
   \\
   & & - (n+3L-1)-(n+3L) +(n+3L +1)+(n+4L) =0, \\
      \text{wt}_f (u_4) &=& -\sum_{i \in \{1,4\}} \left(\sum_{j=1} ^{L/2 -1} (n+(i-1)L+2j-1) +\sum_{j=1} ^{L/2 -1} (n+(i-1)L+2j) \right)
 \\
& &   + \sum_{i=2}^3 \left(\sum_{j=1} ^{L/2 -1} (n+(i-1)L+ 2j-1) +\sum_{j=1} ^{L/2 -1} (n+(i-1)L+2j) \right),   \end{eqnarray*}thus $\text{wt}_f (u_4) =0$. By definition,  $f$ is a DDM labeling for $\ar{G}$, and  $\ar{H}$ is a $3-$nodes ornament. 

{\texttt{Case s=4:}} 
Let $L\geq 7$. We denote four $L-$cycles by ${H_j}$, $1\leq j\leq 4$,  with vertex sets $V(H_j)=\{v_i^{(j)} \; | \; 1 
 \leq i \leq  L\}$.  
 We set an orientation on each $H_j$ by defining the edge sets as follows:
 For each $j\in \{0,3\}$, let 
 
\vspace{1.3mm}
 
\noindent $\displaystyle E(\ar{H}_{j+1})  = \{ v_L^{(j)} v_1^{(j)}, v_1^{(j)} v_2^{(j)}, v_2^{(j)} v_4^{(j)}, v_4^{(j)} v_3^{(j)}, v_3^{(j)} v_5^{(j)} \}\cup \{ v_i^{(j)} v_{i+1}^{(j)} \}_{i=5}^{L-1}$; 

\noindent for $j\in \{1, 2\}$, let
 
\noindent   $\displaystyle E(\ar{H}_{j+1})  = \{   v_1^{(j)}v_L^{(j)},   v_2^{(j)}v_1^{(j)}, v_4^{(j)}v_2^{(j)},  v_3^{(j)}v_4^{(j)},   v_5^{(j)}v_3^{(j)} \}\cup \{   v_{i+1}^{(j)}v_i^{(j)} \}_{i=5}^{L-1}$.

\hspace{1mm}

 \noindent Let $H=\cup_{j=1}^4 {H_j}$. We introduce a labeling on $V(H)$: For all $j \in \{1,2,3,4\}$,  let $h(v_i^{(j)})=n+jL+i$,   $1\leq i \leq L$. An illustration of the defined labeling and orientations is seen in  Figure~\ref{4nodeAttachGen}; observe  that the values of the weight $\text{wt}_h$ are $\pm 1, \pm 2, \pm 3$ or $\pm(L-2)$. It is easy to see that if  $L\geq 6$, then there will be exactly four distinct weight values  of $\text{wt}_h$. 


Recall,  we assume that $\ar{G_0}$ is a DDMOG with $n$ vertices, and associated DDM labeling $f_0$.  In addition, we assume that $n\geq L-2$ as per Table~\ref{s-table}. Let  $f_0(u_1)=1$, $f_0(u_2)=2$, $f_0(u_3)=3$ and $f_0(u_{L-2})=L-2$, where $V_u=\{u_1, u_2, u_3, u_{L-2} \}$ $\subset V({G_0})$. We now introduce additional edges to or from  the vertices in $V_u$ and any vertex $v$ in $  V({H})$  as follows: Given  $p \in \{ 1, 2, 3, L-2\} $, let 
    $vu_{p} \in E'$ if $ \text{wt}_h{(v)}=p$; let $u_{p}v \in E'$ if $ \text{wt}_h{(v)}=-p$; otherwise no new edge is created. 

\begin{figure}
\centering
\begin{tikzpicture}
[scale=0.94, node distance=1.4cm, font=\scriptsize, 
        align=center, base node/.style={circle,draw,minimum size=36pt}]

\node[base node, very thick, label={[text=red]left:+(L-2)}](n+1) at (-8, 2.6) {$v_{1}^{(0)}$};
\node[base node, very thick, label={[text=red] above right:-3}] (n+2) at (-5.5, 2.6) { $v_{2}^{(0)}$};
\node[base node, very thick, label={[text=red]above right:-1}] (n+4) at (-3, 2.6) { $v_{4}^{(0)}$};
\node[base node, very thick, label={[text=red]above right:-1}] (n+3) at (-.5, 2.6) {$v_{3}^{(0)}$};
\node[base node, very thick, label={[text=red]above right:-3}] (n+5) at (2, 2.6) {$v_{5}^{(0)}$}; 
\node[base node, very thick, label={[text=red] above right:-2}](n+6) at (2, 0) {$v_{6}^{(0)}$};
\node[base node, very thick, label={[text=red] above right:-2}] (n+7) at (-.5, 0) {$v_{7}^{(0)}$};
\node[base node, very thick, label={[text=red] left:+(L-2)}] (n+L) at (-8, 0) {$v_{L}^{(0)}$};
\node[base node, very thick, label={[text=red] above right:-2}] (n+L-1) at (-5.5, 0) {$v_{L-1}^{(0)}$};
\node[ very thick] (dd1) at (-3, 0) {...};

\path[->, very thick] (n+1) edge (n+2);
\path[->, very thick] (n+2) edge (n+4);
\path[->, very thick] (n+4) edge (n+3);
\path[->, very thick] (n+3) edge (n+5);
\path[->, very thick] (n+5) edge (n+6);
\path[->, very thick] (n+6) edge (n+7);
\path[->, very thick] (n+7) edge (dd1);
\path[->, very thick] (dd1) edge (n+L-1);
\path[->, very thick] (n+L-1) edge (n+L);
\path[->, very thick] (n+L) edge (n+1);

\node[base node, very thick, label={[text=red]left:-(L-2)}](n+L+1) at (-8, -2.6) {$v_{1}^{(1)}$};
\node[base node, very thick, label={[text=red] above right:+3}] (n+L+2) at (-5.5, -2.6) {$v_{2}^{(1)}$};
\node[base node, very thick, label={[text=red]above right:+1}] (n+L+4) at (-3, -2.6) {$v_{4}^{(1)}$};
\node[base node, very thick, label={[text=red]above right:+1}] (n+L+3) at (-.5, -2.6) {$v_{3}^{(1)}$};
\node[base node, very thick, label={[text=red]above right:+3}] (n+L+5) at (2, -2.6) {$v_{5}^{(1)}$}; 
\node[base node, very thick, label={[text=red] above right:+2}](n+L+6) at (2, -5.2) {$v_{6}^{(1)}$};
\node[base node, very thick, label={[text=red] above right:+2}](n+L+7) at (-.5, -5.2) {$v_{7}^{(1)}$};
\node[base node, very thick, label={[text=red] left:-(L-2)}] (n+2L) at (-8, -5.2) {$v_{L}^{(1)}$};
\node[base node, very thick, label={[text=red] above right:+2}] (n+2L-1) at (-5.5, -5.2) {$v_{L-1}^{(1)}$};
\node[ very thick] (dd12) at (-3, -5.2) {...};

\path[->, very thick] (n+L+1) edge (n+2L);
\path[->, very thick] (n+2L) edge (n+2L-1);
\path[->, very thick] (n+2L-1) edge (dd12);
\path[->, very thick] (dd12) edge (n+L+7);
\path[->, very thick] (n+L+7) edge (n+L+6);
\path[->, very thick] (n+L+6) edge (n+L+5);
\path[->, very thick] (n+L+5) edge (n+L+3);
\path[->, very thick] (n+L+3) edge (n+L+4);
\path[->, very thick] (n+L+4) edge (n+L+2);
\path[->, very thick] (n+L+2) edge (n+L+1);
 
\node[base node, very thick, label={[text=red]left:-(L-2)}](n+2L+1) at (-8, -7.8) {$v_{1}^{(2)}$};
\node[base node, very thick, label={[text=red] above right:+3}] (n+2L+2) at (-5.5, -7.8) { $v_{2}^{(2)}$};
\node[base node, very thick, label={[text=red]above right:+1}] (n+2L+4) at (-3, -7.8) {$v_{4}^{(2)}$};
\node[base node, very thick, label={[text=red]above right:+1}] (n+2L+3) at (-.5, -7.8) { $v_{3}^{(2)}$};
\node[base node, very thick, label={[text=red]above right:+3}] (n+2L+5) at (2, -7.8) { $v_{5}^{(2)}$}; 
\node[base node, very thick, label={[text=red] above right:+2}](n+2L+6) at (2, -10.4) {$v_{6}^{(2)}$}; 
\node[base node, very thick, label={[text=red] above right:+2}](n+2L+7) at (-.5, -10.4) { $v_{7}^{(2)}$};
\node[base node, very thick, label={[text=red] left:-(L-2)}] (n+3L) at (-8, -10.4) { $v_{L}^{(2)}$};
\node[base node, very thick, label={[text=red] above right:+2}] (n+3L-1) at (-5.5, -10.4) { $v_{L-1}^{(2)}$};
\node[ very thick] (dd13) at (-3, -10.4) {...};

\path[->, very thick] (n+2L+1) edge (n+3L);
\path[->, very thick] (n+3L) edge (n+3L-1);
\path[->, very thick] (n+3L-1) edge (dd13);
\path[->, very thick] (dd13) edge (n+2L+7);
\path[->, very thick] (n+2L+7) edge (n+2L+6);
\path[->, very thick] (n+2L+6) edge (n+2L+5);
\path[->, very thick] (n+2L+5) edge (n+2L+3);
\path[->, very thick] (n+2L+3) edge (n+2L+4);
\path[->, very thick] (n+2L+4) edge (n+2L+2);
\path[->, very thick] (n+2L+2) edge (n+2L+1);

\node[base node, very thick, label={[text=red]left:+(L-2)}](n+3L+1) at (-8, -13) {$v_{1}^{(3)}$};
\node[base node, very thick, label={[text=red] above right:-3}] (n+3L+2) at (-5.5, -13) { $v_{2}^{(3)}$};
\node[base node, very thick, label={[text=red]above right:-1}] (n+3L+4) at (-3, -13) { $v_{4}^{(3)}$};
\node[base node, very thick, label={[text=red]above right:-1}] (n+3L+3) at (-.5, -13) {$v_{3}^{(3)}$};
\node[base node, very thick, label={[text=red]above right:-3}] (n+3L+5) at (2, -13) { $v_{5}^{(3)}$}; 
\node[base node, very thick, label={[text=red] above right:-2}](n+3L+6) at (2, -15.6) { $v_{6}^{(3)}$};  
\node[base node, very thick, label={[text=red] above right:-2}](n+3L+7) at (-.5, -15.6) { $v_{7}^{(3)}$};
\node[base node, very thick, label={[text=red] left:+(L-2)}] (n+4L) at (-8, -15.6) { $v_{L}^{(3)}$};
\node[base node, very thick, label={[text=red] above right:-2}] (n+4L-1) at (-5.5, -15.6) { $v_{L-1}^{(3)}$};
\node[ very thick] (dd14) at (-3, -15.6) {...};

\path[->, very thick] (n+3L+1) edge (n+3L+2);
\path[->, very thick] (n+3L+2) edge (n+3L+4);
\path[->, very thick] (n+3L+4) edge (n+3L+3);
\path[->, very thick] (n+3L+3) edge (n+3L+5);
\path[->, very thick] (n+3L+5) edge (n+3L+6);
\path[->, very thick] (n+3L+6) edge (n+3L+7);
\path[->, very thick] (n+3L+7) edge (dd14);
\path[->, very thick] (dd14) edge (n+4L-1);
\path[->, very thick] (n+4L-1) edge (n+4L);
\path[->, very thick] (n+4L) edge (n+3L+1);

\end{tikzpicture}
\caption{$4-$nodes ornament ($4\ar{C_L}$), with the labeling pattern used in the proof of case $s=4$ of Theorem~\ref{thms}. The labeling is defined as $h(v_i^{(j)})=n+jL+i$,   $1\leq i \leq L$, $0\leq j \le 3$. The values of $\text{wt}_h$ are shown in red;  the nodes $u_1$, $u_2$, $u_3$ and $u_{L-2}$ satisfy $f_0(u_p)=p$, for $p \in \{1,2,3,L-2\}$.  }
\label{4nodeAttachGen}
\end{figure}
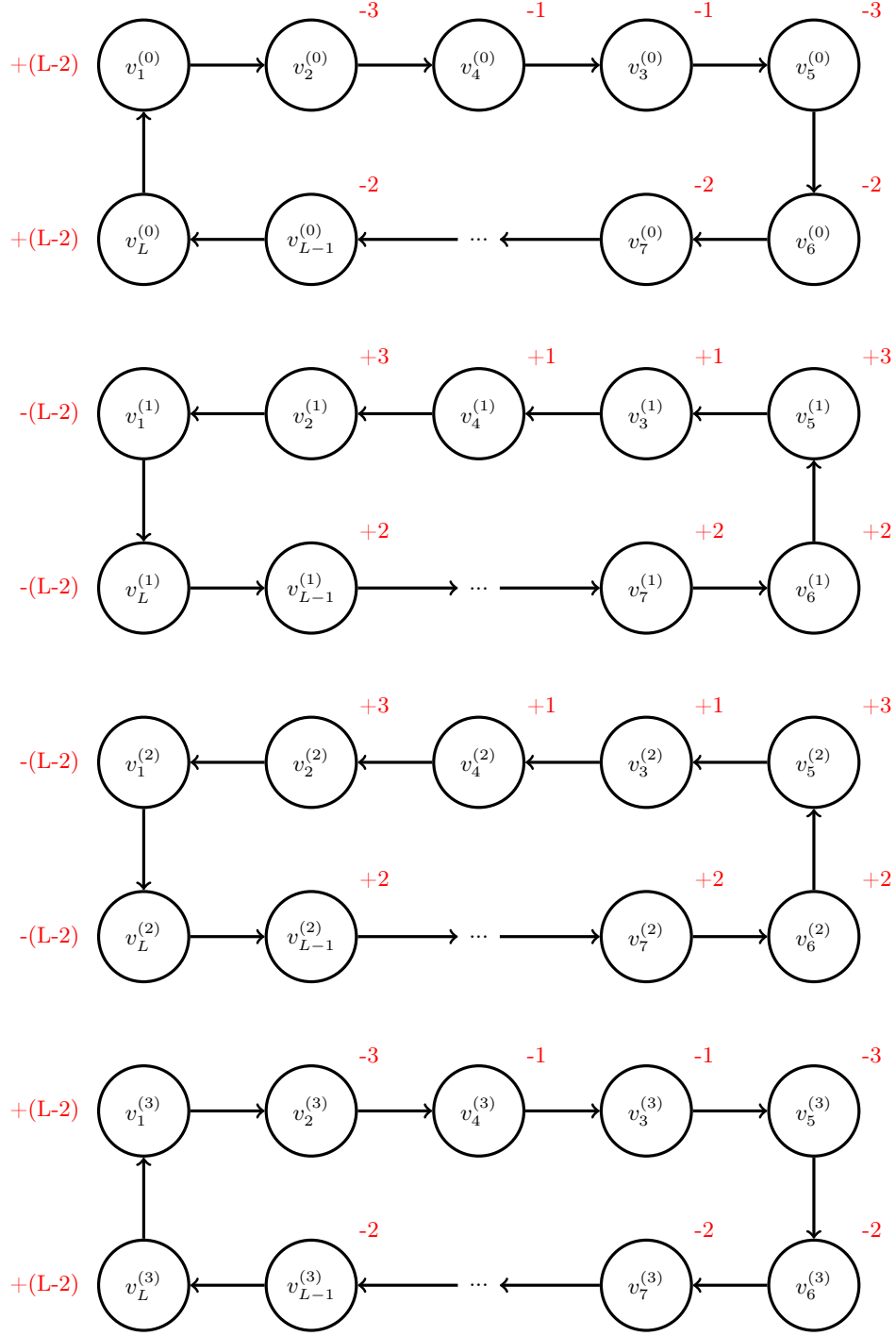\FloatBarrier

We  define a  bijection from $V({G_0}) \cup    V({H}) $ to $\{j\}_{j=1}^{n+4L}$ by $ f(v)=  
        f_0(v) \; \text{if } \; v \in V({G_0})$,   {otherwise,} $f(v)=     h(v)$.  
We  verify by definition that $f$ is indeed a DDM labeling for   $\ar{G}:=\ar{G_0}\oplus_h  \ar{H}$, with associated edges set $E(\ar{G_0})\cup E' \cup \left( \cup_{i=1}^4 E(\ar{H_i})\right)$, and that will complete the proof: First, observe that  $\text{wt}_f$ is zero at every vertex in $ V({H})$, due to the new edges collected in the set $E'$. Also, when $j \notin \{1,2,3,L-2\}$, at every vertex $u_j$ $\in V({G_0})$ we have $\text{wt}_f(u_j)=\text{wt}_{f_0}(u_j)=0$. Finally,
\begin{eqnarray*}
    \text{wt}_f (u_1)&=& \text{wt}_{f_0}(u_1) -(n+4)-(n+3)+(n+L+4)+(n+L+3)\\
  & &  + (n+2L+4)+(n+2L+3) (n+3L+4)-(n+3L+3)=0,
 \\
    \text{wt}_f(u_2) &=& \text{wt}_{f_0}(u_2)+(n+L+6)+(n+L+7)+ ... +(n+2L)-(n+6) \\ 
  & & -(n+7) -...-(n+L)  +(n+2L+6)+(n+2L+7)\\
  & & +...+(n+3L)-(n+3L+6)
  \\& & -(n+3L+7)-...-(n+4L) \\ & =& 0.5(L-6)((n+L+6+n+2L-1)+(n+2L+6+n+3L-1)\\ & &-(n+6+n+L-1)-(n+3L+6+n+4L-1))\\
   &  =& 0.5(L-6)(4n+8L+10-(4n+8L+10)
    = 0 \\
     \text{wt}_f(u_3)&=&\text{wt}_{f_0}(u_3)+(n+2)+(n+5)+(n+3L+2)+(n+3L+5)\\& &-(n+L+2)-(n+L+5)
    -(n+2L+2)-(n+2L+5) = 0,
 \\
    \text{wt}_f(u_{L-2})& =&\text{wt}_{f_0}(u_L-2)+(n+1)+(n+L)+(n+3L+1)+(n+4L)\\
   & & -(n+L+1)-(n+2L)-(n+2L+1)-(n+3L) = 0.
\end{eqnarray*} 
With that, we have confirmed that $\ar{G_0}\oplus_h \cup_{i=1}^4 \ar{H_i}$ is a DDMOG, thus $ \cup_{i=1}^4 \ar{H_i}$ is a $4-$nodes ornament. \qed

{\texttt{Case s=5:}} 
    Let $L\geq 7$ and let $\ar{G_0}$ be an oriented graph  with $n\ge L-2$ vertices, and associated DDM labeling $f_0$ on $V(G_0)$. 
    Suppose at any $u_{i} \in V(G_0)$ we have  $f_0(u_{i}) =i$,  $1\leq i \leq n$. Let  
    ${H_j}$, $1 \leq j \leq 4$ be four copies of an $L-$cycle, with vertex sets  
    $V(H_{j})=  \{v_i^{(j-1)} \}_{i=1}^L$, and let $H=\cup_{j=1}^4 H_j$.
   Next, we define the oriented edges sets as depicted in Figure~\ref{5nodeweights}: 
    for $j \in \{0,3\}$ we have 

    
     \noindent     $ \displaystyle E(\ar{H}_{j+1}) = \{ v_{L}^{(j)} v_{1}^{(j)},  v_{1}^{(j)} v_{3}^{(j)},  v_{3}^{(j)}v_{2}^{(j)},   v_{2}^{(j)}v_{4}^{(j)}\} \cup \{ v_{i}^{(j)} v_{i+1}^{(j)} \}_{i=4}^{ L-1}$,
     
     \noindent while for $j\in \{1,2\}$ we have 
       
    \noindent        $ \displaystyle E(\ar{H }_{j+1}) = \{ v_{L}^{(j)}v_1^{(j)},   v_{3}^{(j)}v_{1}^{(j)},  v_{2}^{(j)}v_{3}^{(j)},   v_{4}^{(j)}v_{2}^{(j)}\} \cup \{ v_{i+1}^{(j)} v_{i}^{(j)} \}_{i= 4}^{ L-1}$.

    
     We  work with the natural bijection $h: \cup_{j=1}^4 V(H_j) \rightarrow \{n+1, n+2, \hdots, n+4L\}$:
     \begin{equation}\label{hlabel5}
         \text{$h(v_i ^{(j)})= n+jL + i$ for  $1\leq i \leq L$, $0\leq j \leq 3$.}
     \end{equation} Depending on the orientation   in each cycle,  the following values for $\text{wt}_h$ will be associated with the vertices of $H$ for $0
     \leq j \leq 3$:   

      \noindent     
 $\text{wt}_h(v_{4}^{(j)})=\pm 3$; it is easy to see that at any other  $v \in V(H)$ it holds   $\text{wt}_h (v)=\pm 2$. Since the maximal value of the weight $\text{wt}_h$ is $\pm (L-2)$, we need $L-2 \leq n$. There will be five distinct nodes with labels $1$, $2$, $3$, $L-3$, and $L-2$ if in adition $L\geq 7$; 
  if $L< 7$, there will be four or less distinct weight values when using the labeling $h$ as defined in \eqref{hlabel5}. 

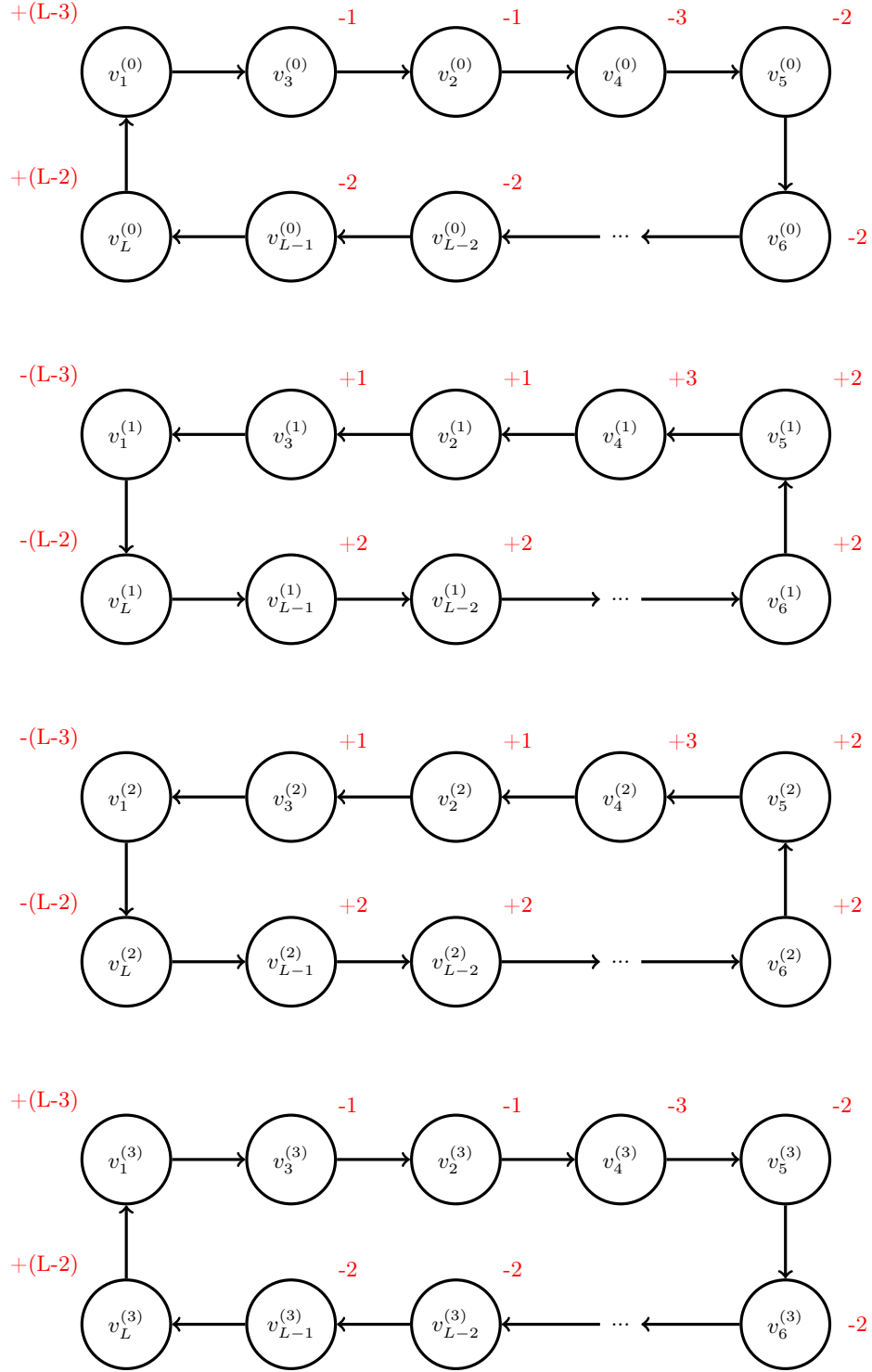
\begin{figure}[h]
\centering
\begin{tikzpicture}
 [scale=0.94, node distance=1.4cm, font=\scriptsize, 
        align=center, base node/.style={circle,draw,minimum size=36pt}]

\node[base node, very thick, label={[text=red]above left:+(L-3)}](31) at (-8, -13) {$v_{1}^{(3)}$};
  
\node[base node, very thick, label={[text=red] above right:-1}] (33) at (-5.5, -13) {$v_{3}^{(3)}$};
 
 \node[base node, very thick, label={[text=red]above right:-1}] (32) at (-3, -13) {$v_{2}^{(3)}$};

  \node[base node, very thick, label={[text=red]above right:-3}] (34) at (-.5, -13) {$v_{4}^{(3)}$};
 
  \node[base node, very thick, label={[text=red]above right:-2}] (35) at (2, -13) {$v_{5}^{(3)}$};

 \node[base node, very thick, label={[text=red]right:-2}](36) at (2, -15.5) {$v_{6}^{(3)}$};
  
\node[ very thick ] (37) at (-.5, -15.5) {...};
 
 \node[base node, very thick, label={[text=red] above right:-2}] (38) at (-3, -15.5) {$v_{L-2}^{(3)}$};

  \node[base node, very thick, label={[text=red]above right:-2}] (39) at (-5.5, -15.5) {$v_{L-1}^{(3)}$};

    \node[base node, very thick, label={[text=red]above left: +(L-2)}] (40) at (-8, -15.5) {$v_{L}^{(3)}$};

\path[->, very thick] (31) edge (33);
 
\path[->, very thick] (33) edge (32);
 
\path[->, very thick] (32) edge (34);

\path[->, very thick] (34) edge (35);
 
\path[->, very thick] (35) edge (36);

\path[->, very thick] (36) edge (37);

\path[->, very thick] (37) edge (38);

\path[->, very thick] (38) edge (39);
 
\path[->, very thick] (39) edge (40);

 \path[->, very thick] (40) edge (31);
 
\node[base node, very thick, label={[text=red]above left:+(L-3)}](1) at (-8, 3.5) {$v_{1}^{(0)}$};
\node[base node, very thick, label={[text=red] above right:-1}] (3) at (-5.5, 3.5) { $v_{3}^{(0)}$};
\node[base node, very thick, label={[text=red]above right:-1}] (2) at (-3, 3.5) {$v_{2}^{(0)}$};
\node[base node, very thick, label={[text=red]above right:-3}] (4) at (-.5, 3.5) {$v_{4}^{(0)}$};
\node[base node, very thick, label={[text=red]above right:-2}] (5) at (2, 3.5) {$v_{5}^{(0)}$};
\node[base node, very thick, label={[text=red]right:-2}](6) at (2, 1) {$v_{6}^{(0)}$};
\node[ very thick ] (7) at (-.5, 1) {...};
\node[base node, very thick, label={[text=red] above right:-2}] (8) at (-3, 1) {$v_{L-2}^{(0)}$};
\node[base node, very thick, label={[text=red]above right:-2}] (9) at (-5.5, 1) {$v_{L-1}^{(0)}$};
\node[base node, very thick, label={[text=red]above left: +(L-2)}] (10) at (-8, 1) {$v_{L}^{(0)}$};

\path[->, very thick] (1) edge (3);
\path[->, very thick] (3) edge (2);
\path[->, very thick] (2) edge (4);
\path[->, very thick] (4) edge (5);
\path[->, very thick] (5) edge (6);

\path[->, very thick] (6) edge (7);

\path[->, very thick] (7) edge (8);

\path[->, very thick] (8) edge (9);
 
\path[->, very thick] (9) edge (10);

 \path[->, very thick] (10) edge (1);
 
\node[base node, very thick, label={[text=red]above left:-(L-3)}](21) at (-8, -7.5) {$v_{1}^{(2)}$};
  
\node[base node, very thick, label={[text=red] above right:+1}] (22) at (-5.5, -7.5) {$v_{3}^{(2)}$};
 
 \node[base node, very thick, label={[text=red]above right:+1}] (23) at (-3, -7.5) { $v_{2}^{(2)}$};

  \node[base node, very thick, label={[text=red]above right:+3}] (24) at (-.5, -7.5) { $v_{4}^{(2)}$};
 
  \node[base node, very thick, label={[text=red]above right:+2}] (25) at (2, -7.5) {$v_{5}^{(2)}$};

 \node[base node, very thick, label={[text=red]above right:+2}](26) at (2, -10) {$v_{6}^{(2)}$};
 
 \node[ very thick ](27) at (-.5, -10) {...};
  
\node[base node, very thick, label={[text=red] above right: +2}] (28) at (-3, -10) { $v_{L-2}^{(2)}$};

 \node[base node, very thick, label={[text=red] above right:+2}] (29) at (-5.5, -10) {$v_{L-1}^{(2)}$};

  \node[base node, very thick, label={[text=red]above left:-(L-2)}] (30) at (-8, -10) {$v_{L}^{(2)}$};

\path[->, very thick] (21) edge (30);
\path[->, very thick] (30) edge (29);
\path[->, very thick] (29) edge (28);
\path[->, very thick] (28) edge (27);
\path[->, very thick] (27) edge (26);
\path[->, very thick] (26) edge (25);
\path[->, very thick] (23) edge (22);
\path[->, very thick] (22) edge (21);
\path[->, very thick] (25) edge (24);
\path[->, very thick] (24) edge (23);

\node[base node, very thick, label={[text=red]above left:-(L-3)}](11) at (-8, -2) {$v_{1}^{(1)}$};
  
\node[base node, very thick, label={[text=red] above right:+1}] (12) at (-5.5, -2)  {$v_{3}^{(1)}$};
 \node[base node, very thick, label={[text=red]above right:+1}] (13) at (-3, -2)  {$v_{2}^{(1)}$};

  \node[base node, very thick, label={[text=red]above right:+3}] (14) at (-.5, -2)  {$v_{4}^{(1)}$};
 
  \node[base node, very thick, label={[text=red]above right:+2}] (15) at (2, -2) {$v_{5}^{(1)}$};

 \node[base node, very thick, label={[text=red]above right:+2}](16) at (2, -4.5)  {$v_{6}^{(1)}$};
 
 \node[ very thick ](17) at (-.5, -4.5) {...};
  
\node[base node, very thick, label={[text=red] above right: +2}] (18) at (-3, -4.5) {$v_{L-2}^{(1)}$};

 \node[base node, very thick, label={[text=red] above right:+2}] (19) at (-5.5, -4.5)  {$v_{L-1}^{(1)}$};

  \node[base node, very thick, label={[text=red]above left:-(L-2)}] (20) at (-8, -4.5) {$v_{L}^{(1)}$};

\path[->, very thick] (11) edge (20);
 
\path[->, very thick] (20) edge (19);
 
\path[->, very thick] (19) edge (18);

\path[->, very thick] (18) edge (17);
 
\path[->, very thick] (17) edge (16);

\path[->, very thick] (16) edge (15);

\path[->, very thick] (13) edge (12);

\path[->, very thick] (12) edge (11);
 
\path[->, very thick] (15) edge (14);
\path[->, very thick] (14) edge (13);

\end{tikzpicture}
\caption{$5-$nodes ornament ($4\ar{C_L}$). The labeling used in the proof of case $s=5$ of Theorem \ref{thms}  is defined by equation \eqref{hlabel5}. 
The values of $\text{wt}_h$ are shown in red, indicating the nodes $u_1$, $u_2$, $u_3$, $u_{L-2}$ and $u_{L-3}$ must satisfy $f_0(u_p)=p$, $p\in\{1,2,3,L-2,L-3\}$.}
\label{5nodeweights}
\end{figure}\FloatBarrier

   
\noindent     $\text{wt}_h(v_{L}^{(j)}) = \pm (L-2)$,      $\text{wt}_h(v_{ 1}^{(j)})=\pm (L-3)$,   $\text{wt}_h (v_{2}^{(j)})=$$\text{wt}_h (v_{3}^{(j)})=\pm 1$,


  Let $V_u = \{u_1, u_2, u_3, u_{L-3}, u_{L-2} \}\subset V(\ar{G_0})$. Now, we define additional oriented edges between the vertices in $V_u$ and the vertices in $V(H_i)$, $1\leq i \leq 4$, and we collect these new edges in the set $E'$. For $p \in \{1,2,3,L-3,L-4\}$, let  
      $v u_p \in E'$ if $\text{wt}_h{(v)}=p$; also, $ u_pv \in E'$ if $\text{wt}_h{(v)}=-p$. 
With that we have completed defining the edge set  $E (\ar{G}) =E' \cup E(\ar{G_0}) \cup (\cup E(\ar{H_i} ))
$.    
Finally, we define a labeling function $f$   as usual: $f|_{V({G_0}) } \equiv f_0$, otherwise $f \equiv h$; clearly, $f$ is a bijection. 

We will now verify by definition that $f$ is a DDM labeling for $\ar{G}$: It is trivial to see that $\text{wt}_f(v)=0 $ for all $v\in \cup_{i=1}^4  V(\ar{H_i})$; further on, if $u\in V(
   \ar{G_0})\setminus V_u$, then $\text{wt}_f(u)=\text{wt}_{f_0}(u)=0$. In addition: 
\begin{eqnarray*}
    \text{wt}_f(u_1)&=& -(n+3)-(n+2) - (n+3L+3)-(n+3L+2)\\ & &+(n+2L+3)+(n+2L+2) +(n+L+3)+(n+L+2) = 0, \\
    \text{wt}_f(u_2) &=& (n+2L-1)+(n+2L-2)+...+(n+L+6)+(n+L+5)\\& &+(n+3L-1)+(n+3L-2)+...+(n+2L+6)+(n+2L+5)\\& &-(n+5)-(n+6)-...-(n+L-2)-(n+L-1)
      \\& &-(n+3L+5)-(n+3L+6)-...-(n+4L-2)-(n+4L-1)
\\
                &=&0.5 (L-5)((n+2L-1+n+L+5)+(n+3L-1+n+2L+5)\\
 & &   -(n+5+n+L-1)-(n+3L+5+n+4L-1))\\
    &=& 0.5 (L-5)(4n+8L+8-(4n+8L+8)) = 0, \\
          \text{wt}_f(u_3) &=& (n+L+4)+(n+2L+4)-(n+4)-(n+3L+4) = 0, \\
   \text{wt}_f(u_{L-3} )&=&  (n+1) + (n+3L+1) - (n+L+1)- (n+2L+1) = 0, \\
    \text{wt}_f(u_{L-2} )&=&  (n+L) + (n+4L) - (n+2L)- (n+3L)  = 0.
   \end{eqnarray*}
Thus, $\ar{G}$ is a DDMOG, and $\ar{H}$ is a $5-$nodes ornament.
   
    {\texttt{Case s=6:}} 
    Let $L\ge 8$ and let $\ar{G_0}$ be a DDMOG with $n\ge L-2$ vertices, and related DDM labeling  $f_0$ on $V(G_0)$ for the underlying graph $G_0$. Suppose at any $u_p\in V(G_0)$ we have $f_0(u_p)=p, 1\le p\le n$. We work with four copies of an L-cycle $H_j$, $1\le j\le 4$, with vertex sets 
    $V(H_{j})=  \{v_i^{(j-1)} \}_{i=1}^L$.
Next, we define the oriented edges  as shown in  Figure~\ref{6nodeAttachGen}.
 \begin{comment}
    
  \noindent       $ \displaystyle E(\ar{H_1}) = \{ v_{n+L} v_{n+1},  v_{n+1} v_{n+3},  v_{n+3}v_{n+4},   v_{n+4}v_{n+2}, v_{n+2}v_{n+5}\} \cup \{ v_{i} v_{i+1} \}_{i=n+ 5}^{ n+L-1}$,    
       
       \noindent    and      $ \displaystyle E(\ar{H_2}) = E_2 \cup E_2'$,  $ \displaystyle E(\ar{H_3}) = E_3 \cup E_3'$,   and      $\displaystyle E(\ar{H_4}) = E_4 \cup E_4'$,    where
 
  \noindent $E_2= \{ v_{i+1} v_{i} \}_{i=n+L+ 5}^{ n+2L-1}$, $E_3=\{ v_{i+1} v_{i} \}_{i=n+2L+ 5}^{ n+3L-1}$, $E_4=\{   v_{i} v_{i+1}\}_{i= n+3L+5}^{n+4L-1}$,  while
            
            \noindent $ \displaystyle E_3'= \{ v_{n+2L+1} v_{n+3L},   v_{n+2L+3}v_{n+2L+1},  v_{n+2L+4}v_{n+2L+3},   v_{n+2L+2}v_{n+2L+4},
v_{n+2L+5}v_{n+2L+2}\}$,      

\noindent $E_4 ' = \{ v_{n+4L} v_{n+3L+1},  v_{n+3L+1} v_{n+3L+3},  v_{n+3L+3}v_{n+3L+4},   v_{n+3L+4}v_{n+3L+2},
             v_{n+3L+2}v_{n+3L+5}\}$.

              \vspace{2.13mm}
\end{comment}
 The labeling $h: V(H) \rightarrow \{n+1, \hdots, n+4L\}$ is the bijection defined by  $h(v_{i}^{(j)}) = n+jL +i$, $1\leq i \leq L$, $0\leq j \leq 3$. Depending on the orientation  in each cycle (see Figure~\ref{6nodeAttachGen}), we have  the following values for $\text{wt}_h$:    
$\pm 1$, $\pm 2$, $\pm 3$, $\pm 4$, $\pm (L-3)$, $\pm (L-2)$.   Since the maximal  value of  $\vert \text{wt}_h \vert$ is $L-2$,  we need $L-2 \leq n$. There will be six distinct nodes with labels $1$, $2$, $3$, $4$, $L-3$, and $L-2$ if $L\geq 8$;   if $L< 8$, some of the node labels will be repeated.

 Now, we define additional oriented edges   {between the nodes} $u_p \in V(G_0)$,  $ p \in \{ 1, 2, 3, 4, {L-3}, {L-2} \}$,   and any vertex $v$ in $V(H)$; we  collect these new edges in  set $E'$.      As usual, for $p \in \{1,2,3,4,L-3,L-2\}$, let
   $v u_p \in E'$  if $\text{wt}_h(v)=p$; if $\text{wt}_h(v)=-p$, then let  $ u_p v \in E'$.

 With that we have completed defining the edge set  of $\ar{G}$. 
 Finally, we define a function $f$ in the usual way.
  Due to the new edges collected in $E'$, we have that $\text{wt}_f(v)=0 $ for all $v\in   V({H})$. Observe that if $p\notin  \{1, 2, 3, 4, {L-3}, {L-2}\}$, then $\text{wt}_f(u_p)=\text{wt}_{f_0}(u_p)=0$. In addition: 

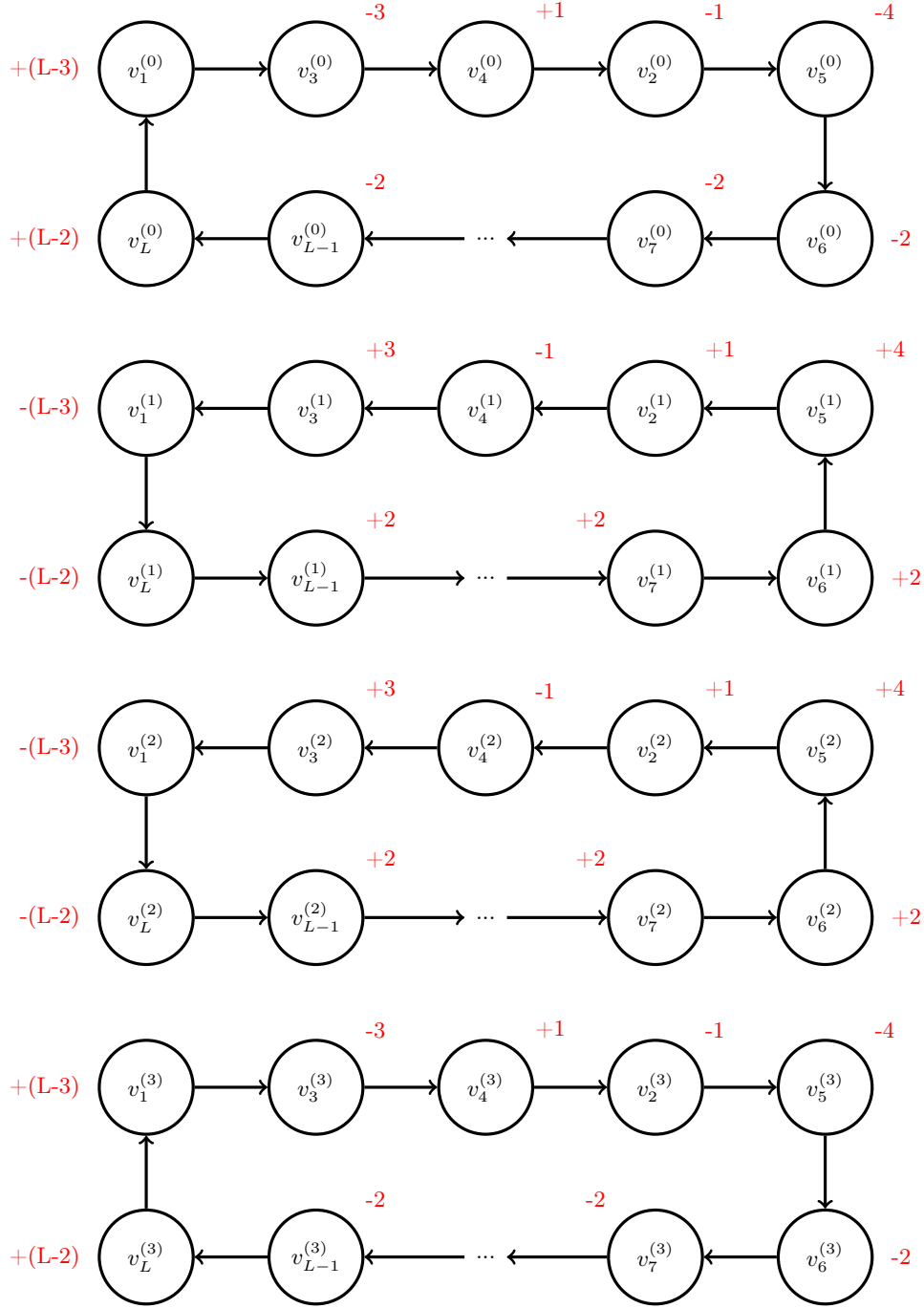
\begin{figure}
\centering
\begin{tikzpicture}
[scale=0.94, node distance=1.4cm, font=\scriptsize, 
        align=center, base node/.style={circle,draw,minimum size=37pt}]
\node[base node, very thick, label={[text=red]left:+(L-3)}](n+1) at (-8, 2.5) {$v_{1}^{(0)}$};
  
\node[base node, very thick, label={[text=red] above right:-3}] (n+3) at (-5.5, 2.5) {$v_{3}^{(0)}$};
 \node[base node, very thick, label={[text=red]above right:+1}] (n+4) at (-3, 2.5) { $v_{4}^{(0)}$};

  \node[base node, very thick, label={[text=red]above right:-1}] (n+2) at (-.5, 2.5) { $v_{2}^{(0)}$};

  \node[base node, very thick, label={[text=red]above right:-4}] (n+5) at (2, 2.5) {$v_{5}^{(0)}$}; 
  
  \node[base node, very thick, label={[text=red]right:-2}] (n+6) at (2,0) {$v_{6}^{(0)}$};

 \node[base node, very thick, label={[text=red] above right:-2}](n+7) at (-.5, 0) { $v_{7}^{(0)}$};
  
\node[base node, very thick, label={[text=red] left:+(L-2)}] (n+L) at (-8, 0) {$v_{L}^{(0)}$};
 \node[base node, very thick, label={[text=red] above right:-2}] (n+L-1) at (-5.5, 0) {$v_{L-1}^{(0)}$};

  \node[ very thick] (dd1) at (-3, 0) {...};

\path[->, very thick] (n+1) edge (n+3);
 
\path[->, very thick] (n+3) edge (n+4);
 
\path[->, very thick] (n+4) edge (n+2);

\path[->, very thick] (n+2) edge (n+5);
 
\path[->, very thick] (n+5) edge (n+6);

\path[->, very thick] (n+6) edge (n+7);

\path[->, very thick] (n+7) edge (dd1);

\path[->, very thick] (dd1) edge (n+L-1);


\path[->, very thick] (n+L-1) edge (n+L);
 
\path[->, very thick] (n+L) edge (n+1);

\node[base node, very thick, label={[text=red]left:-(L-3)}](n+L+1) at (-8, -2.5) {$v_{1}^{(1)}$};
  
\node[base node, very thick, label={[text=red] above right:+3}] (n+L+3) at (-5.5, -2.5) {$v_{3}^{(1)}$};
 \node[base node, very thick, label={[text=red]above right:-1}] (n+L+4) at (-3, -2.5) {$v_{4}^{(1)}$};

  \node[base node, very thick, label={[text=red]above right:+1}] (n+L+2) at (-.5, -2.5)  {$v_{2}^{(1)}$};

  \node[base node, very thick, label={[text=red]above right:+4}] (n+L+5) at (2, -2.5)  {$v_{5}^{(1)}$};
  
  \node[base node, very thick, label={[text=red]right:+2}] (n+L+6) at (2, -5)  {$v_{6}^{(1)}$};

 \node[base node, very thick, label={[text=red] above left:+2}](n+L+7) at (-.5, -5)  {$v_{7}^{(1)}$};
  
\node[base node, very thick, label={[text=red] left:-(L-2)}] (n+2L) at (-8, -5)  {$v_{L}^{(1)}$};
 \node[base node, very thick, label={[text=red] above right:+2}] (n+2L-1) at (-5.5, -5)  {$v_{L-1}^{(1)}$};

  \node[ very thick] (dd12) at (-3, -5) {...};

\path[->, very thick] (n+L+1) edge (n+2L);
 
\path[->, very thick] (n+2L) edge (n+2L-1);
 
\path[->, very thick] (n+2L-1) edge (dd12);

 
\path[->, very thick] (dd12) edge (n+L+7);

\path[->, very thick] (n+L+7) edge (n+L+6);

\path[->, very thick] (n+L+6) edge (n+L+5);

\path[->, very thick] (n+L+5) edge (n+L+2);

\path[->, very thick] (n+L+2) edge (n+L+4);
 
\path[->, very thick] (n+L+4) edge (n+L+3);
  
\path[->, very thick] (n+L+3) edge (n+L+1);
 
\node[base node, very thick, label={[text=red]left:-(L-3)}](n+2L+1) at (-8, -7.5)  {$v_{1}^{(2)}$};
  
\node[base node, very thick, label={[text=red] above right:+3}] (n+2L+3) at (-5.5, -7.5)   {$v_{3}^{(2)}$};
 
 \node[base node, very thick, label={[text=red]above right:-1}] (n+2L+4) at (-3, -7.5)   {$v_{4}^{(2)}$};

  \node[base node, very thick, label={[text=red]above right:+1}] (n+2L+2) at (-.5, -7.5)   {$v_{2}^{(2)}$};

  \node[base node, very thick, label={[text=red]above right:+4}] (n+2L+5) at (2, -7.5)   {$v_{5}^{(2)}$};
  
  \node[base node, very thick, label={[text=red]right:+2}] (n+2L+6) at (2, -10)   {$v_{6}^{(2)}$};

 \node[base node, very thick, label={[text=red] above left:+2}](n+2L+7) at (-.5, -10)   {$v_{7}^{(2)}$};
  
\node[base node, very thick, label={[text=red] left:-(L-2)}] (n+3L) at (-8, -10)   {$v_{L}^{(2)}$};

 \node[base node, very thick, label={[text=red] above right:+2}] (n+3L-1) at (-5.5, -10)  {$v_{L-1}^{(2)}$};

  \node[ very thick] (dd13) at (-3, -10) {...};

\path[->, very thick] (n+2L+1) edge (n+3L);
 
\path[->, very thick] (n+3L) edge (n+3L-1);
 
\path[->, very thick] (n+3L-1) edge (dd13);

\path[->, very thick] (dd13) edge  (n+2L+7);

\path[->, very thick] (n+2L+7) edge (n+2L+6);

\path[->, very thick] (n+2L+6) edge (n+2L+5);

\path[->, very thick] (n+2L+5) edge (n+2L+2);

\path[->, very thick] (n+2L+2) edge (n+2L+4);
 
\path[->, very thick] (n+2L+4) edge (n+2L+3);
  
\path[->, very thick] (n+2L+3) edge (n+2L+1);

\node[base node, very thick, label={[text=red]left:+(L-3)}](n+3L+1) at (-8, -12.5)   {$v_{1}^{(3)}$};
  
\node[base node, very thick, label={[text=red] above right:-3}] (n+3L+3) at (-5.5, -12.5)   {$v_{3}^{(3)}$};

 \node[base node, very thick, label={[text=red]above right:+1}] (n+3L+4) at (-3, -12.5)   {$v_{4}^{(3)}$};

  \node[base node, very thick, label={[text=red]above right:-1}] (n+3L+2) at (-.5, -12.5)   {$v_{2}^{(3)}$};

  \node[base node, very thick, label={[text=red]above right:-4}] (n+3L+5) at (2, -12.5)   {$v_{5}^{(3)}$}; 
  
  \node[base node, very thick, label={[text=red]right:-2}] (n+3L+6) at (2, -15)  {$v_{6}^{(3)}$};

 \node[base node, very thick, label={[text=red] above left:-2}](n+3L+7) at (-.5, -15)   {$v_{7}^{(3)}$};
  
\node[base node, very thick, label={[text=red] left:+(L-2)}] (n+4L) at (-8, -15)  {$v_{L}^{(3)}$};
 \node[base node, very thick, label={[text=red] above right:-2}] (n+4L-1) at (-5.5, -15)   {$v_{L-1}^{(3)}$};

  \node[ very thick] (dd14) at (-3, -15) {...};

\path[->, very thick] (n+3L+1) edge (n+3L+3);
 
\path[->, very thick] (n+3L+3) edge (n+3L+4);
 
\path[->, very thick] (n+3L+4) edge (n+3L+2);

\path[->, very thick] (n+3L+2) edge (n+3L+5);
 
\path[->, very thick] (n+3L+5) edge (n+3L+6);

\path[->, very thick] (n+3L+6) edge (n+3L+7);

\path[->, very thick] (n+3L+7) edge (dd14);

\path[->, very thick] (dd14) edge (n+4L-1);

\path[->, very thick] (n+4L-1) edge (n+4L);
 
\path[->, very thick] (n+4L) edge (n+3L+1);

\end{tikzpicture}
\caption{ $6-$nodes ornament, $4\ar{C_L}$. The labeling  used in the proof of Theorem \ref{thms}, case $s=6$, is defined  as $h(v_{i}^{(j)}) = n+jL +i$, $1\leq i \leq L$, $0\leq j \leq 3$. As usual, the values of $\text{wt}_h$ are shown in red, indicating the nodes $u_p$ must satisfy $f_0(u_p)=p$,  $p \in \{1,2,3,4, L-3, L-2\}$. }
\label{6nodeAttachGen}
\end{figure}\FloatBarrier

   \begin{eqnarray*}
        \text{wt}_f(u_1)&=&    (n+4)+(n+L+2)+(n+2L+2)+(n+3L+4)\\ & &
    -((n+2)+(n+L+4)+(n+2L+4)+(n+3L+2))\\
   & =&4n +6L +12-(4n +6L+12)
    = 0 \\
    \text{wt}_f(u_2) &=& 
    n+L+6+n+L+7+...+n+2L-2+n+2L-1\\
  & & +  n+2L+6+n+2L+7+...+n+3L-2+n+3L-1\\
    & &-(n+6+n+7+...+n+L-2+n+L-1\\
  & & +  n+3L+6+n+3L+7+...+n+4L-2+n+4L-1)\\
    & =&0.5 (L-6)(n+L+6+n+2L-1)+(n+2L+6+n+3L-1)\\
    & &-(n+6+n+L-1)-(n+3L+6+n+4L-1)\\
    & =& 0.5(L-6)(4n+8L+10-(4n+8L+10))
    = 0,\\
 \text{wt}_f(u_3) &=& 
    (n+L+3)+(n+2L+3)-((n+3)+(n+3L+3))=0,
     \\
        \text{wt}_f(u_4) & =&  
    (n+L+5)+(n+2L+5)-((n+5)+(n+3L+5))
       =0,\\
    \text{wt}_f(u_{L-3} )&=&       (n+1) + (n+3L+1) - ((n+L+1)+ (n+2L+1) ) =0,
       \\
    \text{wt}_f(u_{L-2} )  &= &  
    (n+L) + (n+4L) - ((n+2L)+ (n+3L)) 
 \\  &=& 2n + 5L - (2n +5L) = 0.
    \end{eqnarray*}

%
 
%
\noindent At all $v \in V(G)$, it holds  $\text{wt}_f(v)=0$. Thus,  $\ar{G}$ is a DDMOG, and $\ar{H}$ is a $6-$nodes ornament. 
\end{proof}
\subsection{Proof of Theorem~\ref{thm:zigzag} and related examples:} \label{thm:zigzagproof}

    Before diving into the proof of Theorem~\ref{thm:zigzag}, we first review two examples that showcase how the maximum number of  nodes  differs depending on whether $L$ is odd or even.  
Sufigure~\ref{Hep(L-1)}   contains  an example of  the case  when $L=4k+7$  (specifically, here $L=7$), where an ornament is presented with labeling that leads to a DDMOG, after attaching it to any DDMOG $\ar{G_0}$ with $n=6$ vertices.  After introducing additional edges to and from the nodes $u_i$, $1\leq i \leq 6$, with labels $f_0(u_i)=i$, it is easy to confirm that $ \text{wt}_f(u_i)=0$, $1\leq i \leq 6$. 
For an example when $L$ is even  we point to  Figure \ref{Oct(L-2)} ($L=8$). Note that in this case we again have six nodes even though the value of $L$ has increased from $7$ to $8$. 

\begin{proof}[Proof of  Theorem~\ref{thm:zigzag}]
Let $H_i=C_L$, $i \in \{1,2,3,4\}$ be four copies of an  $L-$cycle, with vertex sets $V( {H_i})=\{v_{i,j}\}_{j=1}^L$. We introduce an orientation on  $H =\cup_{i=1}^4 H_i$ by defining  the oriented edges sets in the following way:  
     
      $E(\ar{H_i}) =\{v_{i,1}v_{i,2}, v_{i,2}v_{i,3}, \hdots, v_{i,(L-1)}v_{i,L},v_{i,L}v_{i,1}\}$, for $i \in \{1, 4\}$,
     
    $E(\ar{H_i}) =\{v_{i,2}v_{i,1}, v_{i,3}v_{i,2}, \hdots, v_{i,L}v_{i,(L-1)},v_{i,1}v_{i,L}\}$, for $i \in \{2, 3\}$. 

    The approach in the proof of of Theorem~\ref{thm:zigzag}  varies related to the value of $L$ ($\text{mod} 4$); thus we will need to look at four separate cases (Cases I - IV). The proofs of cases I and II are very similar, with only minor differences; cases III and IV are also very similar to one another.

 \texttt{Case I:} Let $L=4k+7$, $k\geq 0$.   Table~\ref{maxnodelabels} contains a schematics of the  labeling function $h$ we use in this proof. We  compute the initial  weight values at every vertex (with the sign changing from $+$ to $-$ depending on the orientation of each cycle $\ar{H_i}$, $1\leq i \leq 4$:

 \begin{eqnarray*}
   c_1:=  \text{wt}_h(v_{i,1})  &=& \pm [ (\frac{3L+3}{4})-(3+\frac{L-7}{4})] = \pm\frac{L-1}{2}, \\
   c_2:=  \text{wt}_h(v_{i,2})&=&\pm [L-1], \\
   c_3:=  \text{wt}_h(v_{i,3})&=&\pm [ (3+\frac{L-7}{4})-(5+3(\frac{L-7}{4}))] = \pm\frac{L-3}{2},\\
c_4:= \text{wt}_h(v_{i,4})&=&\pm [ L-2], 
\\
&\vdots &\\
c_{{\small L-1}}:= \text{wt}_h(v_{i, {  L-1}})&=& \pm[\lceil\frac{L}{2}\rceil-1]=\pm\frac{L-1}{2}.
\end{eqnarray*}


\begin{table}[H]
    \centering
    \begin{tabular}{c|c|c|c|c|c|c|c|c|c}
        L=4k+8 &  $v_{2j}$ & $v_2$ &$v_4$&$v_6$&...&$v_L$ \\
         & h($v_{2j}$)& $\frac{L}{4}+1$&$3+3(\frac{L-4}{4})$ &$\frac{L}{4}+2$&...&$\frac{L}{2}+1$\\
        &&&&&& \\
        &$v_{2j-1}$ & $v_1$&$v_3$&$v_5$&...&$v_{L-1}$ \\
        & h($v_{2j-1}$)& 1&$L$&2&...&$\frac{3L+4}{4}$\\
       \hline
        L=4k+7 &  $v_{2j}$& $v_2$ &$v_4$&$v_6$&...&$v_{L-1}$\\
       & h($v_{2j}$)& $3+(\frac{L-7}{4})$&$5+3(\frac{L-7}{4})$&$4+(\frac{L-7}{4})$&...&$\lceil\frac{L}{2}\rceil$\\
       &&&&&& \\
       &$v_{2j-1}$& $v_1$&$v_3$&$v_5$&...&$v_{L}$\\
       &h($v_{2j-1}$)& 1&$L$&2&...&$\frac{3L+3}{4}$\\
       \hline
        L=4k+6 & $v_{2j}$& $v_2$ &$v_4$&$v_6$&...&$v_L$\\
        &h($v_{2j}$)& $5+3(\frac{L-6}{4})$&$3+(\frac{L-6}{4})$&$4+3(\frac{L-6}{4})$&...&$\frac{L}{2}+1$\\
         &&&&&& \\
        &$v_{2j-1}$& $v_1$&$v_3$&$v_5$&...&$v_{L-1}$\\
        &h($v_{2j-1}$)& 1&$L$&2&...&$\frac{L+2}{4}$\\
       \hline
        L=4k+5 & $v_{2j}$& $v_2$ &$v_4$&$v_6$&...&$v_{L-1}$ \\
        &h($v_{2j}$)& $4+3(\frac{L-5}{4})$&$\lceil\frac{L}{4}\rceil+1$&$3+3(\frac{L-5}{4})$&...&$\lceil\frac{L}{2}\rceil$\\
         &&&&&& \\
        &$v_{2j-1}$& $v_1$&$v_3$&$v_5$&...&$v_{L}$\\
        &h($v_{2j-1}$)& 1&$L$&2&...&$\lceil\frac{L}{4}\rceil$\\
    \end{tabular}
    \caption{Vertex labels for max nodes ornaments}
    \label{maxnodelabels}
\end{table} 

Observe that the total number of different wight values listed above is $L-1$; thus, we work with $L-1$ nodes in the DDMOG $\ar{G_0}$, $u_1, \hdots, u_{L-1}$, which have labels  $c_1, c_2, \hdots, c_{L-1}$. Here we work with  $\vert V({G_0})\vert =n\geq L-1$, and a   a DDM labeling $f_0$ of $\ar{G_0}$.
Next, we introduce a labeling function  $f:V(\ar{G_0}\oplus_{\text{wt}_h} (\bigcup_{i=1}^4 \ar{H_i}))\rightarrow\{1,...,n,...,n+4L\}$  such that 
  $f(v) =       f_0(v)$ when $v \in V(G_0)$; otherwise, 
     $f(v)= h(v)$.  It is clear that $f$ is a bijection. Now, we introduce additional edges from or to  each  vertex in $V(H_i)$, $1\leq i \leq 4$,  with the intention to neutralize the weights at each vertex; we collect these new edges in a new  set $E'$. For example, if a vertex in $V(H_i)$ has a weight of $\text{wt}_{h}(v)=1$, we add the edge $vu_1$ to $E'$. Further, if a vertex in $V(H_i)$ has a weight of $\text{wt}_{h}(v)=-1$, add the edge $u_1v$ to $E'$. We repeat this for every $v \in \cup_{i=1}^4 V(H_i)$ by attaching the vertex to the node labeled with its weight. With that we have completed defining the edge set  $E (\ar{G}) =E' \cup E(G_0) \cup \left( \cup E(\cup_{i=1}^4\ar{H_i} )\right)
$. With that,  it holds $\text{wt}_f(v)=0$  at every $v \in V(\cup_{i=1}^4{H_i} )$. 

It will  show  that the value of $\text{wt}_f$ at every node is $0$; {verifying this is an easy exercise for the reader.}
With that, we have verified the theorem for the case $L=4k+7$, $k\geq 0$.

\vspace{2.3mm}

\texttt{Case II:} $L=4k+5$, $k\geq 0$; we leave the details to the reader.

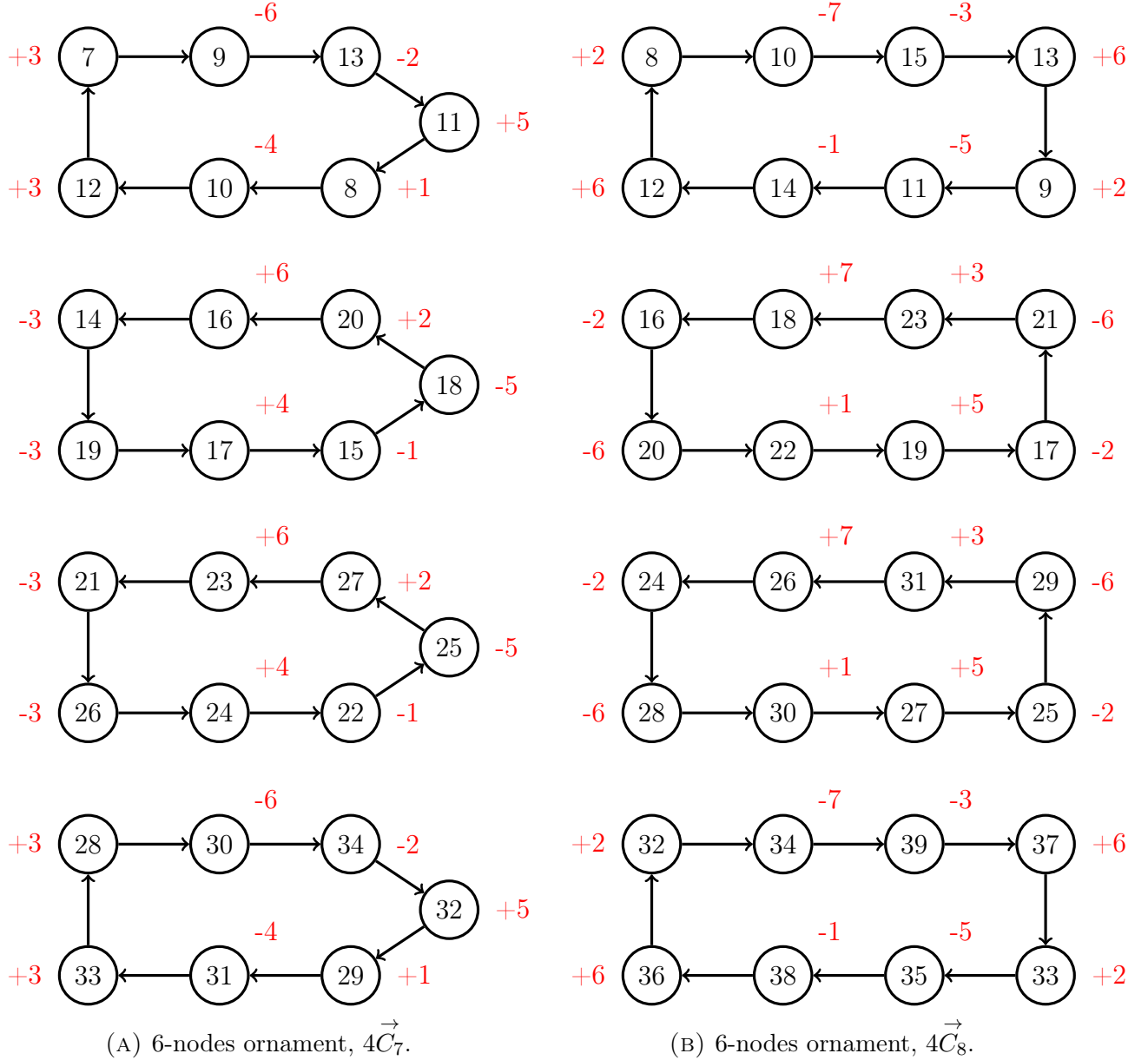
\begin{figure}[htbp]
\centering
\begin{subfigure}[t]{0.48\textwidth}
\centering
\begin{tikzpicture}
[node distance=1.4cm, base node/.style={circle,draw,minimum size=25pt}]

\node[base node, very thick, label={[text=red] left:+3}] (7) at (-5, 2) {7};
\node[base node, very thick, label={[text=red]above right:-6}] (9) at (-3, 2) {9};
\node[base node, very thick, label={[text=red]right:-2}] (13) at (-1, 2) {13};
\node[base node, very thick, label={[text=red]right:+5}] (11) at (0.5, 1) {11};
\node[base node, very thick, label={[text=red]right:+1}](8) at (-1, 0) {8};
\node[base node, very thick, label={[text=red] left:+3}] (12) at (-5, 0) {12};
\node[base node, very thick, label={[text=red]above right:-4}] (10) at (-3, 0) {10};

\path[->, very thick] (7) edge (9);
\path[->, very thick] (9) edge (13);
\path[->, very thick] (13) edge (11);
\path[->, very thick] (11) edge (8);
\path[->, very thick] (8) edge (10);
\path[->, very thick] (10) edge (12);
\path[->, very thick] (12) edge (7);

\node[base node, very thick, label={[text=red] left:-3}] (14) at (-5, -2) {14};
\node[base node, very thick, label={[text=red]above right:+6}] (16) at (-3, -2) {16};
\node[base node, very thick, label={[text=red]right:+2}] (20) at (-1, -2) {20};
\node[base node, very thick, label={[text=red]right:-5}] (18) at (0.5, -3) {18};
\node[base node, very thick, label={[text=red]right:-1}](15) at (-1, -4) {15};  
\node[base node, very thick, label={[text=red] left:-3}] (19) at (-5, -4) {19};
\node[base node, very thick, label={[text=red]above right:+4}] (17) at (-3, -4) {17};

\path[->, very thick] (14) edge (19);
\path[->, very thick] (19) edge (17);
\path[->, very thick] (17) edge (15);
\path[->, very thick] (15) edge (18);
\path[->, very thick] (18) edge (20);
\path[->, very thick] (20) edge (16);
\path[->, very thick] (16) edge (14);

\node[base node, very thick, label={[text=red] left:-3}] (21) at (-5, -6) {21};
\node[base node, very thick, label={[text=red]above right:+6}] (23) at (-3, -6) {23};
\node[base node, very thick, label={[text=red]right:+2}] (27) at (-1, -6) {27};
\node[base node, very thick, label={[text=red]right:-5}] (25) at (0.5, -7) {25};
\node[base node, very thick, label={[text=red]right:-1}](22) at (-1, -8) {22};
\node[base node, very thick, label={[text=red] left:-3}] (26) at (-5, -8) {26};
\node[base node, very thick, label={[text=red]above right:+4}] (24) at (-3, -8) {24};

\path[->, very thick] (21) edge (26);
\path[->, very thick] (26) edge (24);
\path[->, very thick] (24) edge (22);
\path[->, very thick] (22) edge (25);
\path[->, very thick] (25) edge (27);
\path[->, very thick] (27) edge (23);
\path[->, very thick] (23) edge (21);

\node[base node, very thick, label={[text=red] left:+3}] (28) at (-5, -10) {28};
\node[base node, very thick, label={[text=red]above right:-6}] (30) at (-3, -10) {30};
\node[base node, very thick, label={[text=red]right:-2}] (34) at (-1, -10) {34};
\node[base node, very thick, label={[text=red]right:+5}] (32) at (0.5, -11) {32};
\node[base node, very thick, label={[text=red]right:+1}](29) at (-1, -12) {29};  
\node[base node, very thick, label={[text=red] left:+3}] (33) at (-5, -12) {33};
\node[base node, very thick, label={[text=red]above right:-4}] (31) at (-3, -12) {31};

\path[->, very thick] (28) edge (30);
\path[->, very thick] (30) edge (34);
\path[->, very thick] (34) edge (32);
\path[->, very thick] (32) edge (29);
\path[->, very thick] (29) edge (31);
\path[->, very thick] (31) edge (33);
\path[->, very thick] (33) edge (28);

\end{tikzpicture}
\caption{$6$-nodes ornament, $4\ar{C_7}$. 
}
\label{Hep(L-1)}
 
\end{subfigure}
\hfill
\begin{subfigure}[t]{0.48\textwidth}
\centering
\begin{tikzpicture}
[node distance=1.4cm, base node/.style={circle,draw,minimum size=25pt}]

\node[base node, very thick, label={[text=red] left:+2}] (8) at (-5, 2) {8};
\node[base node, very thick, label={[text=red]above right:-7}] (10) at (-3, 2) {10};
\node[base node, very thick, label={[text=red]above right:-3}] (15) at (-1, 2) {15};
\node[base node, very thick, label={[text=red]right:+6}] (13) at (1, 2) {13};
\node[base node, very thick, label={[text=red]right:+2}] (9) at (1, 0) {9};
\node[base node, very thick, label={[text=red]above right:-5}](11) at (-1, 0) {11};
\node[base node, very thick, label={[text=red]above right:-1}] (14) at (-3, 0) {14};
\node[base node, very thick, label={[text=red] left:+6}] (12) at (-5, 0) {12};

\path[->, very thick] (8) edge (10);
\path[->, very thick] (10) edge (15);
\path[->, very thick] (15) edge (13);
\path[->, very thick] (13) edge (9);
\path[->, very thick] (9) edge (11);
\path[->, very thick] (11) edge (14);
\path[->, very thick] (14) edge (12);
\path[->, very thick] (12) edge (8);

\node[base node, very thick, label={[text=red] left:-2}] (16) at (-5, -2) {16};
\node[base node, very thick, label={[text=red]above right:+7}] (18) at (-3, -2) {18};
\node[base node, very thick, label={[text=red]above right:+3}] (23) at (-1, -2) {23};
\node[base node, very thick, label={[text=red]right:-6}] (21) at (1, -2) {21};
\node[base node, very thick, label={[text=red]right:-2}] (17) at (1, -4) {17};
\node[base node, very thick, label={[text=red]above right:+5}](19) at (-1, -4) {19};  
\node[base node, very thick, label={[text=red]above right:+1}] (22) at (-3, -4) {22};
\node[base node, very thick, label={[text=red] left:-6}] (20) at (-5, -4) {20};

\path[->, very thick] (16) edge (20);
\path[->, very thick] (20) edge (22);
\path[->, very thick] (22) edge (19);
\path[->, very thick] (19) edge (17);
\path[->, very thick] (17) edge (21);
\path[->, very thick] (21) edge (23);
\path[->, very thick] (23) edge (18);
\path[->, very thick] (18) edge (16);

\node[base node, very thick, label={[text=red] left:-2}] (24) at (-5, -6) {24};
\node[base node, very thick, label={[text=red]above right:+7}] (26) at (-3, -6) {26};
\node[base node, very thick, label={[text=red]above right:+3}] (31) at (-1, -6) {31};
\node[base node, very thick, label={[text=red]right:-6}] (29) at (1, -6) {29};
\node[base node, very thick, label={[text=red]right:-2}] (25) at (1, -8) {25};
\node[base node, very thick, label={[text=red]above right:+5}](27) at (-1, -8) {27};
\node[base node, very thick, label={[text=red]above right:+1}] (30) at (-3, -8) {30};
\node[base node, very thick, label={[text=red] left:-6}] (28) at (-5, -8) {28};

\path[->, very thick] (24) edge (28);
\path[->, very thick] (28) edge (30);
\path[->, very thick] (30) edge (27);
\path[->, very thick] (27) edge (25);
\path[->, very thick] (25) edge (29);
\path[->, very thick] (29) edge (31);
\path[->, very thick] (31) edge (26);
\path[->, very thick] (26) edge (24);

\node[base node, very thick, label={[text=red] left:+2}] (32) at (-5, -10) {32};
\node[base node, very thick, label={[text=red]above right:-7}] (34) at (-3, -10) {34};
\node[base node, very thick, label={[text=red]above right:-3}] (39) at (-1, -10) {39};
\node[base node, very thick, label={[text=red]right:+6}] (37) at (1, -10) {37};
\node[base node, very thick, label={[text=red]right:+2}] (33) at (1, -12) {33};
\node[base node, very thick, label={[text=red]above right:-5}](35) at (-1, -12) {35};  
\node[base node, very thick, label={[text=red]above right:-1}] (38) at (-3, -12) {38};
\node[base node, very thick, label={[text=red] left:+6}] (36) at (-5, -12) {36};

\path[->, very thick] (32) edge (34);
\path[->, very thick] (34) edge (39);
\path[->, very thick] (39) edge (37);
\path[->, very thick] (37) edge (33);
\path[->, very thick] (33) edge (35);
\path[->, very thick] (35) edge (38);
\path[->, very thick] (38) edge (36);
\path[->, very thick] (36) edge (32);

\end{tikzpicture}
 
\caption{$6$-nodes ornament, $4\ar{C_8}$.}
\label{Oct(L-2)}
\end{subfigure}
\caption{Examples of c-ornaments utilizing the labeling pattern as per the proof of Theorem \ref{thm:zigzag}.}
\end{figure}

\text{Case III:} Let $L=4k+8$, $k\geq 0$. We use the values of the labeling function $h$ as presented in Table~\ref{maxnodelabels} to compute the initial  weights at every vertex  (with the sign changing from $+$ to $-$ depending on the orientation of the cycle $\ar{H_i}$, $1\leq i \leq 4$:

\begin{eqnarray*}
  c_1:=\text{wt}_h(v_{i,1}) &=& \pm [ (\frac{L}{2}+1)-(\frac{L}{4}+1)] = \pm\frac{L}{4},\\
  c_2:=  \text{wt}_h(v_{i,2}) &=& \pm [L-1],  \\
 c_3:=   \text{wt}_h(v_{i,3}) &=& \pm [ (\frac{L}{4}+1)-(3+3(\frac{L-4}{4}))] =\pm\frac{L-2}{2},\\ c_4:=   \text{wt}_h(v_{i,4}) & =&\pm [ L-2], \\
    & \vdots & \\
c_{L-2}:=\text{wt}_h(v_{i,L-2}) &=& \pm[\frac{3L+4}{4}-1]=\pm\frac{3L}{4}.\\
\end{eqnarray*}
    The total number of different values listed above is $L-2$; thus, we work with $L-2$ nodes ($u_{c_j}$, $1\leq j \leq L-2$) in a DDMOG $\ar{G_0}$ with $n\geq L-2$ vertices and DDM labeling $f_0$. 

We work with bijection $f:V(\ar{G_0}\oplus_{\text{wt}_h} (\cup_{i=1}^4 \ar{H_i}))\rightarrow\{1,2, \hdots, n+4L\}$  such that 
  $f(v) =  
      f_0(v)$ {when}  $v \in V(G_0)$; otherwise   
 $f(v) = h(v)$. 
  Now, we define additional edges from/to  each  vertex in $V(H_i)$, $1\leq i \leq 4$,  with the intention to 'neutralize' the weights at each vertex. For example, if a vertex in $V(H_i)$ has a weight of $\text{wt}_{f}(v)=1$, add the edge $vu_1$. Further, if a vertex in $V(H_i)$ has a weight of $\text{wt}_{f}(v)=-1$, add the edge $u_1v$. Repeat this for every $v \in \cup_{i=1}^4 V(H_i)$ by attaching the vertex to the node labeled with its weight. With that we have completed defining the edge set  $E (\ar{G}) =E' \cup V(G_0) \cup \left( \cup_{i=1}^4 E(\ar{H_i} )\right)
$. The added edges (from $E'$) ensure that $\text{wt}_f(v)=0$ for every $v \in V(\cup_{i=1}^4 \ar{H_i})$. 
It it easy to show that the value of $\texttt{wt}_f$ at every node is also  $0$; therefore, $\ar{G}$ is a DDMOG when $L=4k+8, k\geq 0$.  
 

\texttt{Case IV:}   $L=4k+6$, $k\geq 0$;  we leave the details to the reader.
\end{proof}

\section{Conclusion and open questions}\label{summary}

In this paper we have presented ways to create $s-$nodes ornaments of type $4\ar{C_L}$; in the proofs we used  repetitive, consecutive labeling patterna, broadly  described by \eqref{repetPatt}. Theorem \ref{thms}    offers ways to generate $s$-nodes ornaments, where $2\leq s\leq 6$,  with the assumption
$a_s\leq L\le b_s$ (recall, $a_s$, $b_s$ take values as listed in Table~\ref{s-table}).  
Theorem~\ref{thm:zigzag} addresses a maximizing question, where $s$ equals   $(L-1)$ or $(L-2)$, depending on whether $L$ is odd or even.   One possible application of these results could be to build complex networks with the presented method; specifically, multiple cyclic networks could be attached to complete ones (with DDM labeling), without creating structural bottlenecks or areas of chronic surplus. We also note that  Lemma~\ref{minLvsS} indicates that the upper bound for $s$ relative to the value of $L$ may be higher that the one given in Theorem~\ref{thm:zigzag}, and we anticipate that the critical value ($s=L$) may be achieved on $4\ar{C_L}$ if other, non-consecutive labeling patterns are employed.  So, we have left a question not fully answered in this paper: {\it Given $L\geq 3$, what is the maximal value of $s$ so that $4\ar{C_L}$ is a $s-$nodes ornament?}

Other unanswered questions arose as we were working on this paper. One was already posed  in the introduction, motivated by the observation in  Remark~\ref{rem1}: {\it  If $\ar{H}$ is any oriented graph paired with a labeling function $h$ such that         $\text{wt}_h(\ar{H})=0$, does there exist an integer $k>0$ such that $k$ copies of $\ar{H}$  form an ornament?} We list a few more open  questions here, such as:  {\it  
What is the minimal number of vertices in an $s-$nodes ornament?} If we call an ornament a {\it minimal} $s-$nodes ornament when it has the minimal number of vertices among all of the $s-$nodes ornaments, then the question becomes: {\it Given $s$, what is the minimal $s-$node ornament?} In this paper we have only studied ornaments formed by cycles, that is c-ornaments, and  we reflect briefly on what we know about these types of ornaments in term of minimality: When $s=1$, the minimal number of vertices is $4$, as one copy of the $4-$cycle forms a $1-$node c-ornament, \cite{AAetc}.
The sparsest (in terms of number of vertices) $3-$nodes c-ornament discovered so far has  $20$ vertices  (Figure~\ref{fig:pentasimpleInversion}). 
 For  the case when $s=4$, it seems that the minimal c- ornament has again   $20$ vertices (Figure~\ref{fig:pentaInversionAnother}).  

Another question we have encountered as we were writing this paper is related to the minimal number of edges a connected DDMOG with $n$ vertices must have: {\it What is the sparsest connected DDMOG with $n$ vertices in terms of edges?} When observing the  DDMOG with $14$ vertices in Figure \ref{fig:triangles2onsingletons}, we conjecture that it is the sparsest DDMOG with $14$ vertices. 
As mentioned in Remark~\ref{rem24}, the construction of DDMOGs via ornaments  delivers DDMOGs with fewer edges than the construction offered in Theorem 4.5 in \cite{AAetc}, but we cannot claim that we have found the minimal DDMOG in terms of number of edges for the case $n=14$.    In addition, a similar question for c-ornaments opens up (but no longer requires the assumption of connectedness): {\it What is the sparsest $s-$node ornament with a total of $k$ vertices?}


Finally, this paper builds on results from paper \cite{AAetc} and offers a plethora of examples of  DDMOGs created via weighted sums; even A  the wheel can be created as a weighted sum of a single node and a cycle. Then, it is natural to ask: {\it 
Which DDMOGs cannot be created via a weighted sum of a smaller DDMOG with some cycle(s)?}  
We hope the readers will be inspired to pursue the answers to some of the questions we pose.


\end{document}